\documentclass[a4paper,11pt]{article}
\usepackage{fullpage}
\usepackage{amsmath,amssymb,latexsym,amsthm}
\usepackage[pdftex]{graphicx}
\usepackage{amsmath,amssymb}
\usepackage{graphicx}
\usepackage{color, epsfig}
\usepackage{subfigure}
\usepackage{pgf,pgfarrows,pgfnodes,pgfautomata,pgfheaps,pgfshade}
\usepackage{lmodern}

\def\Cr{\color{red}}

\def\Cbe{\color{blue}}

\definecolor{gold}{rgb}{0.85,.66,0}

\definecolor{purple}{rgb}{.403,.067,.53}

\def\card{\#}
\providecommand{\R}{\mathbb{R}}

\providecommand{\N}{\mathbb{N}}

\newtheorem{theorem}{Theorem}
\newtheorem{definition}{Definition}

\newtheorem{proposition}{Proposition}
\newtheorem{lemma}{Lemma}
\newtheorem{remark}{Remark}
\def\Lip{\mbox{\rm Lip}}
\title{Quantitative compactness estimates\\ for Hamilton-Jacobi equations}
\author{
Fabio Ancona\footnote{
Dipartimento di Matematica,
Universit\`a di Padova,
Via Trieste 63, 35121 Padova, Italy
} ,
Piermarco Cannarsa\footnote{
Dipartimento di Matematica,
Universit\`a di Roma 'Tor Vergata, 
Via della Ricerca Scientica 1, 00133 Roma, Italy} ,
Khai T. Nguyen\footnote{
Department of Mathematics, Penn State University,
University Park, Pa. 16802, U.S.A.}
}
\begin{document}
\maketitle
\begin{abstract}
We study quantitative compactness estimates in $\mathbf{W}^{1,1}_{loc}$
for the map $S_t$, $t>0$
that associates to every given initial data $u_0\in \Lip (\mathbb{R}^N)$ the corresponding solution
$S_t u_0$ of  a  Hamilton-Jacobi equation
$$
u_t+H\big(\nabla_{\!x}  u\big)=0\,,
\qquad t\geq 0,\quad x\in \mathbb{R}^N,
$$
with a uniformly convex Hamiltonian $H=H(p)$.
We provide upper and lower estimates of order $1/\varepsilon^N$
on the the Kolmogorov $\varepsilon$-entropy in $\mathbf{W}^{1,1}$ of the image through the map $S_t$
of sets of bounded, compactly supported initial data.
Estimates of this type are inspired by a question posed by P.D. Lax~\cite{Lax02} within the context of
conservation laws, and
could provide a measure of the order of ``resolution'' of a numerical method implemented for this equation.
\end{abstract}




\section{Introduction}
Consider a first-order Hamilton-Jacobi equation 
\begin{equation}
\label{HJ}
u_t(t,x)+H\big(\nabla_{\!x}  u(t,x)\big)=0\,,
\qquad t\geq 0,\quad x\in \mathbb{R}^N,
\end{equation}
where $u=u (t,x)$, $\nabla_{\!x} u =(u_{x_1},\dots, u_{x_N})$,
and $H:\mathbb{R}^N\rightarrow\mathbb{R}$ is a smooth
Hamiltonian.
It is well-known that, because of the nonlinear dependence of the characteristic speeds
on the gradient of the solution,
in general classical solutions $u(t,x)$ of the Cauchy problem for~\eqref{HJ} develop singularities
of $\nabla_{\!x} u(t,x)$ in finite time, no matter how smooth the initial data 
\begin{equation}
\label{in-data}
u(0,\cdot)= u_0
\end{equation}
are assumed to be.
To cope with this difficulty, M.G. Crandall and P.-L. Lions introduced in~\cite{CL} the notion of viscosity solution,
a generalized solution
of~\eqref{HJ}, 
which allows to establish global existence, uniqueness and stability results for the Cauchy problem~\eqref{HJ}-\eqref{in-data},
under suitable  assumption on~$H$.
We refer to~\cite{BCD} for a review of the concept of viscosity solution and the related theory for equation
of type~\eqref{HJ} that has been developed in the last thirty years.
\pagebreak

The Hamiltonian $H$ is required here to satisfy the \textbf{Standing Assumption}:
\begin{itemize}
%
\item[{\bf (H1)}] $H\in C^2(\mathbb{R}^N)$ and is uniformly convex, i.e.
$$D^2H(p)\geqslant\alpha\cdot \mathbb{I}_{N}\qquad\forall p\in\mathbb{R}^N,$$

where $\alpha$ is a positive constant, $\mathbb{I}_{N}$ is the $N\times N$ identity matrix,
and the inequality is understood in the sense that $D^2H(p)-\alpha\cdot \mathbb{I}_{N}$ is a positive semidefinite matrix. 
\end{itemize}
The assumption {\bf (H1)}
guarantees that, if the initial data $u_0:\mathbb{R}^N\to\mathbb{R}$
is Lipschitz continuous and bounded, the Cauchy problem~\eqref{HJ}-\eqref{in-data}
admits a unique viscosity solution $u(t,x)$ which is Lipschitz continuous and 
semiconcave in $x$ 
with semiconcavity constant $1/(\alpha t)$. This means that $x\mapsto u(t,x)-1/(2\alpha t) |x|^2$
is a concave function. In turn, this fact implies that $u(t,\cdot)$ is almost everywhere
twice differentiable and that $\nabla_{\!x} u(t,\cdot)$ has locally bounded variation,
i.e. that the distributional Hessian $D^2_x u(t,\cdot)$ is a symmetric matrix of Radon measures.

Furthermore, one can define a semigroup of viscosity solutions of~\eqref{HJ}
$$\big\{S_t: \Lip (\mathbb{R}^N)\to \Lip(\mathbb{R}^N)\big\}_{t\geqslant 0}$$
that associates to every  initial data $u_0\in \Lip (\mathbb{R}^N)$
the unique viscosity solution $S_t u_0 := u(t,\cdot)$ of the corresponding Cauchy problem~\eqref{HJ}-\eqref{in-data}.
It is not difficult to see that the semigroup map $S_t$ is continuous 
when it is restricted to subsets of $\Lip (\mathbb{R}^N)$ 
bounded  in $\mathbf{W}^{1,\infty}$, taking the $\mathbf{W}^{1,1}_{loc}$-topology 
on $\Lip (\mathbb{R}^N)$
(cfr.~Proposition~\ref{St-W11-continuity} in Section~\ref{sec:prelim}). Moreover,
thanks to the uniform semiconcavity constant of 
$S_t u_0$, for  $u_0\in  \Lip (\mathbb{R}^N)$,  
applying Helly's compactness theorem and a Poincar\'e inequality for BV-functions,
one can show that the restriction of $S_t$, $t>0$, to such sets is compact  with respect to
the  $\mathbf{W}^{1,1}_{loc}$-topology. This property reflects the irreversibility
feature of the equation~\eqref{HJ} when the Hamiltonian $H$ satisfies the convexity assumption~{\bf (H1)}.

The aim of this paper is to provide a quantitative estimate 
of this regularizing effect of the semigroup map.
Namely, having in mind a question posed by P.D. Lax~\cite{Lax02} within the context of conservation laws,
we wish to estimate the Kolmogorov $\varepsilon$-entropy in $\mathbf{W}^{1,1}$ of the image through the map $S_t$
of sets of bounded, compactly supported initial data $\mathcal{C}\subset \Lip (\mathbb{R}^N)$
of the form
\begin{equation}
\label{Cclass}
\mathcal{C}_{[L,M]}:=
\Big\lbrace{u_0\in \Lip(\mathbb{R}^N)\ \big|\ \mathrm{supp}(u_0)\subset [-L,L]^N\,,\; 
\Lip[u_0]\leqslant M\Big\rbrace}.
\end{equation}
Actually, since the solution of the Cauchy problem for~\eqref{HJ} with zero initial data is the function $u(t,x)=-t\cdot H(0)$,
it will be convenient to analyze  the Kolmogorov $\varepsilon$-entropy in $\mathbf{W}^{1,1}$ of the translated set $S_t(\mathcal{C})+ t\cdot H(0)$,
with $\mathcal{C}$ as in~\eqref{Cclass}.
We recall the notion of $\varepsilon$-entropy introduced by A.~Kolmogorov~\cite{KT}:
\begin{definition}\label{DefKE}
Let $(X,d)$ be a metric space and let $K$ be a totally bounded subset of $X$. For $\varepsilon>0$, let $\mathcal{N}_{\varepsilon}(K|X)$  be the minimal number of sets in a cover of $K$ by subsets of $X$ having diameter no larger than $2\varepsilon$. Then the $\varepsilon$-entropy of $K$ is defined as
\begin{equation} \nonumber
\mathcal H_{\varepsilon}(K|X) := \log_{2} \mathcal{N}_{\varepsilon}(K|X).
\end{equation}
Throughout the paper, we will call {\em $\varepsilon$-cover} a cover of $K$ by subsets of $X$ having diameter no larger than $2\varepsilon$.
\end{definition}
Entropy numbers play a central role in various areas of information theory and statistics
as well as of ergodic and learning theory. In the present setting, this concept  could 
provide a measure of the order of ``resolution'' and of the ``complexity'' of a numerical scheme, as suggested in~\cite{Lax78}.
Roughly speaking, the order of magnitude of the $\varepsilon$-entropy
should indicate the minimum number of operations that one should perform 
in order to obtain an approximate solution with a precision of order $\varepsilon$
with respect to the considered topology.

In this paper we provide both upper and lower bounds of order $1/\varepsilon^N$ on the $\varepsilon$-entropy 
in~$\mathbf{W}^{1,1}$ of $S_t(\mathcal{C})+ t\cdot H(0)$, for sets $\mathcal{C}$ as in~\eqref{Cclass}, thus showing that
such an $\varepsilon$-entropy is of size $\approx 1/\varepsilon^N$.
Without loss of generality, we will assume that the Hamiltonian satisfies  further
\begin{itemize}
\item[\bf (H2)]\hspace{6cm} $\nabla H(0)=0$,
\end{itemize}
otherwise the transformations $x\to x + t \nabla H(0)$ and $H(p) \to H(p) - \langle \nabla H(0), p\rangle$
reduce the general case to this one.
%
Specifically, we prove the following
\begin{theorem}
\label{upper-lower-estimate}
Let $H:\R^N\to \R$ be  a function satisfying  the assumptions~{\bf (H1)-(H2)}
and $\{S_t\}_{t\geqslant 0}$ be the semigroup of viscosity solutions generated by~\eqref{HJ}
on the domain $\Lip (\mathbb{R}^N)$. 
Then, given  $L,M,T>0$,  
for every $\varepsilon>0$ sufficiently small
the following estimates hold:
\begin{equation}
\label{Upper-est-H}
\mathcal{H}_{\epsilon}\Big(S_T(\mathcal C_{[L,M]})+T\cdot H(0)\ \big|\ \mathbf{W}^{1,1}(\mathbb{R}^N)\Big)\leqslant \Gamma^+_{[L,M,N,T]}\cdot\frac{1}{\varepsilon^N}
\end{equation} 
with 
\begin{align}
\label{G+def}
\Gamma^+_{[L,M,N,T]}&:= \omega_N^N \cdot\! \bigg(4N\!\cdot\!\Big(1+ M+\big({1}/{(\alpha\,T)}+1\big)\cdot l_{[L,M,T]}\Big)\bigg)^{\!4N^2}
%
\\
\noalign{\medskip}
\label{LT}
l_{[L,M,T]}&:=L+T\cdot\sup_{|p|\leqslant M}|\nabla H(p)|,
\end{align}
$\alpha$ being the constant appearing in {\bf (H1)}
and $\omega_N$ denoting the Lebesgue measure of the unit ball of $\mathbb{R}^N$, and
%
%
\begin{equation}
\label{Lower-est-H}
\mathcal{H}_{\epsilon}\Big(S_T(\mathcal C_{[L,M]})+T\cdot H(0)\ \big|\ \mathbf{W}^{1,1}(\mathbb{R}^N)\Big)\geqslant \Gamma^-_{[L,N,T]}\cdot\frac{1}{\varepsilon^N}
\end{equation} 
with 
\begin{equation}
\label{G-def}
\Gamma^-_{[L,N,T]}:=\frac{1}{8\cdot \ln 2}\Bigg(\frac{\omega_{N}}{192\, (N+1)\cdot\|D^2H(0)\|\cdot T} \Bigg)^{\!\!\!N}
\cdot\bigg(\frac{L}{4}\bigg)^{\!\!\!N(N+1)}.
\end{equation}
%
\end{theorem}
\medskip
In the one dimensional case ($N=1$) the above estimates can be easily obtained 
recalling the well-known fact (e.g. see~\cite{KR}) that $u(t,x)$ is a viscosity solution of~\eqref{HJ} if and only if its space derivative
$v(t,x):=u_x(t,x)$ is an entropy weak solution of the conservation law 
\begin{equation}
\label{ConLaw}
v_t +H(v)_x=0,
\end{equation}
and relying on the same type of estimates established in~\cite{AON1,DLG}
for scalar conservation laws.
In fact, denoting with ${\widetilde S}_t$  the semigroup map generated by~\eqref{ConLaw},
observe that any $\varepsilon$-cover  in $\mathbf{W}^{1,1}$ for a translated set $S_t(\mathcal{C})+ t\cdot H(0)$
of solutions  to~\eqref{HJ} at time $t$, with initial data in  $\mathcal{C}$, provides also an $\varepsilon$-cover  in~$\mathbf{L}^1$
for the  set ${\widetilde S}_t(\mathcal{C}')$ of solutions to~\eqref{ConLaw} at time $t$,
with initial data in  $\mathcal{C}':=\{u'\,|\, u\in\mathcal{C}\}$.
Thus, applying~\cite[Thorem 1.3]{AON1}
one derives the lower bound 
$\mathcal{H}_{\epsilon}(S_t(\mathcal{C})+ t\cdot H(0)\, | \, \mathbf{W}^{1,1})\geqslant \mathcal{H}_{\epsilon}(S_t(\mathcal{C}')\, | \, \mathbf{L}^1)
\underset{\approx}>$
 $\frac{L^2}{|H''(0)| \cdot t}\cdot\frac{1}{\varepsilon}$, which is of the same size as the one provided by 
 $\Gamma^-_{[L,1,t]}\cdot\frac{1}{\varepsilon}$
in~\eqref{Lower-est-H}.
On the other hand, invoking a Poincar\'e inequality, one can easily adapt the construction performed in~\cite{DLG} 
of an $\varepsilon$-cover  in $\mathbf{L}^1$ of ${\widetilde S}_t(\mathcal{C}')$
to produce
an $\varepsilon$-cover  in $\mathbf{W}^{1,1}$ of $S_t(\mathcal{C})+ t\cdot H(0)$ with the same number of elements.
As a consequence, we derive an upper bound on $\mathcal{H}_{\epsilon}(S_t(\mathcal{C})+ t\cdot H(0)\, | \, \mathbf{W}^{1,1})$
of the same order as the one established in~\cite[Thorem 2.2]{DLG} (cfr.~also~\cite[Remark 1.4]{AON1})
which, in turn, is of the same size as the one provided by 
 $\Gamma^+_{[L,M,1,t]}\cdot\frac{1}{\varepsilon}$
in~\eqref{Upper-est-H}.

When the space dimension is greater than one we can no more rely on
the equivalence between the theory of Hamilton-Jacobi equations and that of hyperbolic conservation laws.
Indeed,
 in this case, the gradient of a viscosity solution turns out to be 
(at least formally) a solution of a non-strictly hyperbolic system in several space variables,  while the available compactness estimates for systems  of conservation laws
concern only the
class of strictly hyperbolic systems in one space variable~\cite{AON2, AON3}.
Neverthless, we shall implement some of the ideas originated in the works~\cite{AON1,DLG}
to prove Theorem~\ref{upper-lower-estimate}. 
However, in order to handle the higher dimensional case, one needs  new ideas  which exploit  specific properties of the viscosity solutions of~\eqref{HJ}
as well as the geometrical theory of monotone functions of several variables.

Towards the derivation of the upper bound stated in (i), we observe that 
for any given viscosity solution $u(t,x)$, 
letting $D^+_x u$ denote a generalized space gradient of $u$ (cfr. Definition~\ref{general-grad}),
the semiconcavity property  of $u$ ensures that the map $x\mapsto  D^+_x u(t,x)-\frac{x}{\alpha\,t}$ is a monotone 
decreasing multifunction on $\mathbb{R}^N$. Next, relying on  a Poincar\'e inequality,
we  provide an upper bound on the $\varepsilon$-entropy  in $\mathbf{L}^1$ for a class of 
monotone decreasing multifunctions with uniformly bounded total variation, defined on
a bounded domain of $\mathbb{R}^N$. In turn, such a  bound  yields
 estimate~\eqref{Upper-est-H} on the $\varepsilon$-entropy  in $\mathbf{W}^{1,1}$
of $S_T(\mathcal{C}_{[L,M]})+T\cdot H(0)$,  again by Poincar\'e's inequality.

The lower bounds on $H_{\varepsilon}(S_T(\mathcal{C}_{[L,M]})+T\cdot H(0))$ are obtained in two steps
adopting a similar strategy as the one pursued in~\cite{AON1}. 
\begin{itemize}
\item[1.] 
We consider
a class $\mathcal{SC}_{[K]}$ of semiconcave functions with 
semiconcavity constant $K$,
defined on a bounded domain,
and we establish a controllability type result for the elements of such a class,
up to a translation by a fixed map. 
Namely, employing the Hopf-Lax formula for the viscosity
solutions to~\eqref{HJ}
we prove that, at any given time $T>0$,  every element of $\mathcal{SC}_{[K]}-T\cdot H(0)$
can be obtained
as the value $u(T,\cdot)$ of a classical solution of~\eqref{HJ}, with initial data in 
$\mathcal{C}_{[L,M]}$, provided that the semiconcavity constant $K$
is sufficient small. Since a classical solution must coincide with the unique viscosity solution of the
corresponding Cauchy problem, this proves that $\mathcal{SC}_{[K]} -T\cdot H(0)\subset S_T(\mathcal{C}_{[L,M]})$.
\item[2.] We introduce a one-parameter class of semiconcave functions $\mathcal{U}_n\subset \mathcal{SC}_{[K]}$
defined as combinations of suitable
bump functions and, by a combinatorial argument, we provide an optimal estimate (w.r.t.  parameter $n$)
of the maximum number of functions in $\mathcal{U}_n$ at
distance $\leq \varepsilon$ w.r.t. the $\mathbf{W}^{1,1}$-metric. This estimate yields a lower bound
on the $\varepsilon$-entropy of $\mathcal{U}_n$, from which we recover~\eqref{Lower-est-H}
relying on the result of point 1.
\end{itemize}

The paper is organized as follows. In Section~\ref{sec:prelim}, we collect  preliminary results and definitions
concerning semiconcave functions and  Hamilton-Jacobi equations. 
In Section~\ref{sec:up-ext}, after deriving  further properties of the viscosity solutions of Hamilton-Jacobi equations,
we provide an upper bound on the $\varepsilon$-entropy in $\mathbf{L}^1$ 
for a class of monotone multifunctions. Relying on this result,
we next establish an upper bound on the $\varepsilon$-entropy in $\mathbf{W}^{1,1}$ for a class of semiconcave functions, which yields
the upper bound stated in Theorem~\ref{upper-lower-estimate}-$(i)$. 
In Section~\ref{sec:low-est}, we carry out the analysis described in the above two steps to obtain 
the lower bound stated in Theorem~\ref{upper-lower-estimate}-$(ii)$.
\medskip
\section{Notation and preliminaries}
\label{sec:prelim}

Let $N\geqslant 1$ be an integer. Throughout the paper we shall denote by:
\begin{itemize}
\item $|\cdot|$ the Euclidean norm in $\R^N$,
\item $\langle\cdot,\cdot\rangle$ the Euclidean inner product in $\R^N$,  
\item $[x,y]$ the segment joining two points $x,y\in \mathbb{R}^N$,
\item $B(x_0,r)$  the  open ball of $\R^N$ with radius $r >0$ and centered at $x_0$,
\item $\card (S)$ the number of elements of any finite set $S$,
\item $\mathrm{Vol}(D)$ the Lebesgue measure of a measurable set $D\subset \R^N$,
\item $\omega_N :=\mathrm{Vol}(B(0,1))=\frac{\pi^{N/2}}{\Gamma(N/2+1)}$  the Lebesgue measure of the unit ball of $\mathbb{R}^N$,
\item $\|A\|$ the usual operator norm of the $N\times N$ matrix $A$,
\item $\Lip (\R^N)$ the space of all Lipschitz continuous functions $f:\R^N\to \R$, and by $\Lip[f]$ the Lipschitz seminorm of $f$,
\item $\mathrm{supp}(u)$ the support of $u\in\Lip (\R^N)$, that is, the closure of $\big\{x\in\R^N~|~u(x)\neq 0\big\}$,
\item $\mathbf{L}^{1}(D)$, with $D\subset\R^N$  a  measurable set, the Lebesgue space of all (equivalence classes of) summable functions on $D$, equipped with the usual norm $\|\cdot\|_{\mathbf{L}^{1}(D)}$, 
\item $\mathbf{L}^{\infty}(D)$, with $D\subset\R^N$  a  measurable set,  the space of all essentially bounded functions on~$D$,  and by $\|u\|_{\mathbf{L}^{\infty}(D)}$ the essential supremum of a function $u\in \mathbf{L}^{\infty}(D)$ (we shall use the same symbol in case $u$ is vector-valued), 
\item $\mathbf{W}^{1,1}\big(\Omega)$, with $\Omega$ a convex domain in $\R^N$, 
the  Sobolev space of functions with summable first order distributional derivatives, and by $\|\cdot\|_{\mathbf{W}^{1,1}(\Omega)}$ its norm,
\item $\mathbf{W}^{1,1}_0\big(\Omega)$, with $\Omega$ a convex domain in $\R^N$, 
the  Sobolev space of functions $F\in \mathbf{W}^{1,1}\big(\Omega)$ with zero trace on the boundary $\partial \Omega$,
\item $BV(\Omega,\R^N)$, with $\Omega$ a domain in $\R^N$, the space of all vector-valued functions $F:\Omega\to\R^N$ of bounded variation (that is, all $F\in \mathbf{L}^1(\Omega,\R^N)$ such that the first partial derivatives of $F$ in the sense of distributions are  measures with finite total variation in $\Omega$).
\end{itemize}
Moreover $\lfloor a\rfloor:=\max\{z\in \mathbb{Z}\, | x\leq a\}$  denotes the integer part of $a$.
\medskip

\subsection{Semiconcave  and monotone functions  in~$\R^N$ and Poincar\'e inequalities}
\label{subsec:semiconc-monotone}
We collect here some basic definitions and properties of semiconcave and monotone functions in~$\R^N$ that will be used 
in the paper. We refer to~\cite{CS} and ~\cite{AA} for a general introduction to the respective theories.
\begin{definition}\label{SCL} 
A continuous function $u:\Omega\rightarrow \mathbb{R}$, with $\Omega\subset\mathbb{R}^N$,
 is called {\em semiconcave} if there exists $K>0$ such that 
\begin{equation}\label{eq:semiconcave}
u(x+h)+u(x-h)-2u(x)\leqslant K|h|^2,
\end{equation}
for all $x,h\in\R^N$ such that $[x-h,x+h]\subset \Omega$. 
When this property
holds true, we also say that $u$ is {\em semiconcave in $\Omega$ with constant}  $K$, and call $K$ a {\em semiconcavity constant} for~$u$.  
\begin{itemize}
\item[-]We say that $u$ is {\em semiconvex}  (with constant $-K$) if $-u$ is semiconcave (with constant $K$).
\item[-]We say that $u:\Omega\rightarrow \mathbb{R}$, with $\Omega\subset\mathbb{R}^N$ open, 
is {\em locally semiconcave} (or locally semiconvex) if 
$u$ is semiconcave (semiconvex) in every compact set $A \subset\subset\Omega$.
\end{itemize}
\end{definition}
\begin{remark}\rm
\label{semiconc-prop}
The notion of semiconcavity introduced here is the most commonly used in the literature.
A more general definition of semiconcavity can be found in~\cite{CS}.
 It is easy to see that a function $u$  is semiconcave in $\Omega$ with  constant $K$ if any only if the function 
 $$\widetilde{u}(x)=u(x)-\frac{K}{2}|x|^2\qquad(x\in\Omega)$$ 
 is concave. Moreover, any continuously differentiable map $u:\Omega\rightarrow \mathbb{R}$
 that has a Lipschitz continuous gradient $\nabla u$
 with Lipschitz constant $K$ is semiconcave with constant $2K$.
 \end{remark}
 Semiconcave functions share some  well-know properties of concave
 functions (see \cite[Theorem 2.1.7, Theorem 2.3.1]{CS} and \cite[Proposition~5711]{AA}) stated in the following
\begin{theorem}
\label{pro-semiconc}
 Let $\Omega\subseteq\mathbb{R}^N$ open
 and $u:\Omega\rightarrow \mathbb{R}$ be locally semiconcave. Then, the following properties hold true: 
\begin{itemize}
\item [(i)] u is locally Lipschitz continuous.
\item [(ii)] (Alexandroff's Theorem) u is  almost everywhere twice differentiable.
\item[(iii)] The gradient of $u$, defined almost everywhere in $\Omega$,  belongs to $BV_{loc}(\Omega,\mathbb{R}^N)$. 
Moreover, if $u$ is semiconcave in $\Omega$ with constant  $K$, then 
\begin{equation}
D^2 u \leqslant K\cdot \mathbb{I}_{N}\mathcal{L}^N
\end{equation}
in the sense of symmetric matrix-valued measures.
\end{itemize}
\end{theorem}

We shall adopt the notation $Du$
for the distributional gradient of a semiconcave function $u$.
A notion of generalized gradient that is specially fit to viscosity solutions is recalled in the following
\begin{definition}
\label{general-grad}
Let $u:\Omega\rightarrow\mathbb{R}$, with $\Omega\subseteq\mathbb{R}^N$ open. For every $x\in\Omega$, the sets
\begin{equation}
\label{sup-sub-diff}
\begin{aligned}
D^+u(x)
&:=\bigg\lbrace{p\in\mathbb{R}^N\ |\ \limsup_{y\rightarrow x}\frac{u(y)-u(x)-\langle p,y-x\rangle}{|y-x|}\leqslant 0 \bigg\rbrace},
\\
\noalign{\medskip}
D^-u(x)
&:=\bigg\lbrace{p\in\mathbb{R}^N\ |\ \liminf_{y\rightarrow x}\frac{u(y)-u(x)-\langle p,y-x\rangle}{|y-x|}\geqslant 0 \bigg\rbrace},
\end{aligned}
\end{equation}
are called, respectively, the \textit{superdifferential} 
and the \textit{subdifferential} 
of $u$ at $x$. Moreover, 
\begin{equation}
\label{reach-grad}
D^*u(x):=\Big\lbrace{p=\lim_{k\rightarrow\infty}\nabla u(x_k)\ |\ f\ \mathrm{is\ differentiable\ at}\ x_k\ \mathrm{and}\ x_k\rightarrow x\Big\rbrace},
\end{equation}
is called the set of reachable gradients of $u$ at $x$.
\end{definition}
\noindent
From  definition~\eqref{sup-sub-diff} it follows that there holds
\begin{equation}
\label{sup-sub-diff-equal}
D^-u(x)=-D^+(-u)(x)\qquad\quad\forall~x\in\Omega.
\end{equation}
The superdifferential of a semiconcave function enjoys the properties stated in the following (see~\cite[Proposition~3.3.4, Theorem~3.3.6]{CS})
\begin{theorem}
\label{Pro-Semi}
 Let $\Omega\subseteq\mathbb{R}^N$ open
 and $u:\Omega\rightarrow \mathbb{R}$ be locally semiconcave. Then, the following properties hold true.
\begin{enumerate}
\item [(i)] The superdifferential $D^{+}u(x)$ is a compact, convex,
nonempty set for  all $x\in\Omega$.
\item[(ii)] $D^{+}u$ is an upper semicontinuous set-valued map, that is, if $\{x_k\}$ is a sequence in $\Omega$
converging to $x$, and if $p_k\in D^+ u(x_k)$ converges to a vector $p\in\R^N$, then $p\in D^+u(x)$.
\item [(iii)] $D^{+}u(x)=\mathrm{co}\,D^{*}u(x)$ for  all $x\in\Omega$, where {\em co} stands for the convex hull.
\item [(iv)] $D^{+}u(x)$ is a singleton if and only if $u$ is differentiable at $x$.
\item[(v)] If $D^{+}u(x)$ is a singleton for every $x\in\Omega$, 
then $u\in C^{1}(\Omega,\mathbb{R})$. 
\end{enumerate}
\end{theorem}
\begin{remark}\rm
\label{semiconc-semiconc}
Relying on the properties of the generalized gradients one can show that if a 
function $u:\Omega\rightarrow\mathbb{R}$ ( $\Omega\subseteq\mathbb{R}^N$ open and convex)
is both semiconcave and semiconvex in $\Omega$
then $u\in C^{1,1}(\Omega,\mathbb{R})$ (see~\cite[Corollary 3.3.8]{CS}).
\end{remark}
In dealing with the map $x \mapsto D^{+}u(x)$ it will be useful to recall
the following notions for set-valued maps.
\begin{definition}
Let $F:\mathbb{R}^N\to2^{\mathbb{R}^N}$ be a multifunction, that is a map that associates with every point $x\in \R^N$
some set $F(x)\subset \R^N$. We say that $F$ is monotone decreasing if 
\begin{equation}
\label{monotone-decr}
\langle v_2-v_1,x_2-x_1\rangle\leqslant 0,\quad\forall x_i\in\mathbb{R}^N, v_i\in F(x_i),i=1,2. 
\end{equation}
%
%
The set 
\begin{equation*}
\text{dom}(F):=\big\lbrace{x\in\mathbb{R}^N~|~F(x)\neq \varnothing\big\rbrace}
\end{equation*}
 is called  the {\em domain} of $F$. 
We say that $F$ is {\em univalued} on some set $A$ if $F(x)$ consists of at most one point for every $x\in A$.
\end{definition}
%
%

As observed in~\cite{AA} (see Corollary~1.3(3) and Remark 2.3),
any monotone decreasing multifunction $F$ is bounded and almost everywhere univalued in
every open set $\Omega\subset\mathbb{R}^N$, which is relatively compact in the interior of dom$(F)$.
Therefore, we may regard the restriction of $F$
to any such open set $\Omega$ as an element of~$\mathbf{L}^{\infty}(\Omega,\mathbb{R}^N)$.
Actually,  in \cite[Proposition~5.1]{AA},  $F$  is shown to be a function of bounded variation on $\Omega$ 
and the following upper bound on the total variation
of its distributional derivative is provided.
\begin{proposition}\label{BV-bound}
Let $F:\mathbb{R}^N\to2^{\mathbb{R}^N}$ be a monotone decreasing multifunction
and $\Omega\subset\mathbb{R}^N$ be an open set, relatively compact in the interior of dom$(F)$.
Then, the restriction of $F$
to $\Omega$ (viewed as an element of $\mathbf{L}^{\infty}(\Omega,\mathbb{R}^N)$) belongs to $BV(\Omega,\mathbb{R}^N)$. Moreover, setting $F(\Omega):=\cup_{x\in\Omega}F(x)$,
there holds
\begin{equation}
|DF|(\Omega)
\leqslant 2^{\frac{N}{2}} N^2 \omega_N\, \big[\mathrm{diam}(\Omega)+\mathrm{diam}(F(\Omega))
\big]^N
\end{equation}
where $|DF|$ is the total variation of the (matrix-valued) Radon measure $DF$,  and
$$
\mathrm{diam}(A):=\sup\big\{|x_{2}-x_{1}|\ |\ x_{i}\in A\big\}\qquad
(A\subset\R^N)\,.
$$
\end{proposition}

We next recall  further  
properties of semiconcave functions and of their superdifferentials 
(see \cite[Theorem 2.1.7, Theorem 2.3.1, 
Proposition~3.3.10]{CS},
\cite[Corollary 1.4]{AA}).
\pagebreak

\begin{proposition}
\label{prop2}
Let $\Omega\subseteq\mathbb{R}^N$ be open convex
 and $u:\Omega\rightarrow \mathbb{R}$ be semiconcave with  constant $K$. 
Then, the following properties hold.
\begin{enumerate}
\item [(i)] For every $x,y\in\Omega$, 
and for any $p_x\in D^+u(x)$, $p_y\in D^+u(y)$, there holds
\[
\big\langle p_y-p_x,y-x\big\rangle\leqslant K\, |y-x|^2.
\]
\item [(ii)] The map $x\mapsto D^+u(x)-K\, x$ is a 
monotone decreasing multifunction.
\end{enumerate}
\end{proposition}
\medskip

We conclude this paragraph recalling two Poincar\'e-type inequalities that will be used in the paper.
The first one is valid for trace-zero $\mathbf{W}^{1,1}$ functions (e.g. see~\cite[Theorem 3 in Section~5.6]{E}),
while  the second one, based on~\cite[Theorem 3.2]{AD} and on \cite[Proposition 3.2.1, Theorem 3.44]{ANP},
is satisfied by BV functions on convex domain.
\begin{theorem}
\label{poincare}
(Poincar\'e inequalities)
Let $\Omega\subseteq\mathbb{R}^N$ be an open, bounded and convex set with Lipschitz boundary.
\begin{itemize}
\item[$(i)$]
If $u\in \mathbf{W}^{1,1}_0(\Omega)$, then
\begin{equation}
\int_{\Omega}|u| dx\leq (\mathrm{Vol}(\Omega))^{\frac{1}{N}} \int_{\Omega}|\nabla u| dx\,.
\end{equation}
\item[$(ii)$]
If $u \in BV(\Omega,\mathbb{R}^N)$, then, letting
\begin{equation*}
u_{\Omega}:= 
\frac{1}{\mathrm{Vol}(\Omega)}\int_{\Omega}  u(x) dx, 
\end{equation*}
denote the mean value of $u$ over $\Omega$, there holds
\begin{equation}
\label{poinc-ineq}
\int_{\Omega}\big|u-u_{\Omega}\big| dx\leq \frac{\mathrm{diam}(\Omega)}{2}\cdot \big|Du\big|(\Omega),
\end{equation}
where $|Du|$ is the total variation of the Radon measure $Du$.
\end{itemize}
\end{theorem}
\bigskip

\subsection{Hamilton-Jacobi equation}
\label{subsec:HJ}
Consider the Hamilton-Jacobi equation~\eqref{HJ} under the assumptions {\bf (H1)-(H2)}.
Observe that the lower bound on the Hessian matrix $D^2H$ given in~{\bf (H1)}
in particular implies the condition:
\begin{enumerate}
\item[{\bf (H1)$'$}]  $H\in C^2(\mathbb{R}^N)$ and is a uniformly convex and coercive map, i.e., $$\lim_{|p|\rightarrow\infty}\frac{H(p)}{|p|}=+\infty.$$ 
\end{enumerate}
Moreover, relying on {\bf (H1)}, we have that
\begin{equation}
\label{DH0-upbound}
\exists~m_0>0\qquad\quad\text{s.t.}\qquad
\sup_{|p|\leqslant m_0}\big\|D^2H(p)\big\|\leqslant 2\, \big\|D^2H(0)\big\|\,.
\end{equation}
%
As we mentioned in the introduction, classical smooth solutions of~\eqref{HJ}  in general break down
and Lipschitz continuous functions that satisfy~\eqref{HJ} almost everywhere 
together with an initial condition~\eqref{in-data} are not unique. To handle this problem, the following concept of  solution 
was introduced in~\cite{CL} (see also \cite{CEL}) so to guarantee global existence and uniqueness results.
\begin{definition}\label{viscosity-solution}(Viscosity solution)
We say that a continuous function $u:[0,T]\times\mathbb{R}^N$ is a viscosity solution of \eqref{HJ} if:
\begin{enumerate}
\item[$\mathrm{(1)}$] u is a viscosity subsolution of \eqref{HJ}, i.e., for every point $(t_0,x_0)\in \,]0,T[\,\times\mathbb{R}^N$ and test function $v\in C^1\big((0,+\infty)\times\mathbb{R}^N\big)$ such that $u-v$ has a local maximum at $(t_0,x_0)$, it holds
\[
v_t(t_0,x_0)+H\big(\nabla_{\!x} v(t_0,x_0)\big)\leqslant 0\,,
\] 
\item[$\mathrm{(2)}$]  u is a viscosity supersolution of (\ref{HJ}), i.e., for every point $(t_0,x_0)\in \,]0,T[\,\times\mathbb{R}^N$ and test function $v\in C^1\big((0,+\infty)\times\mathbb{R}^N\big)$ such that $u-v$ has a local minimum at $(t_0,x_0)$, it holds
\[
v_t(t_0,x_0)+H\big(\nabla_{\!x} v(t_0,x_0)\big)\geqslant 0\,.
\] 
\end{enumerate}
In addition, we say that $u$ is a viscosity solution of the Cauchy problem~\eqref{HJ}-\eqref{in-data} 
if condition~\eqref{in-data} is satisfied in the classical sense.
\end{definition}
\begin{remark}
\label{visos-sol-differentiab}\rm
By the alternative equivalent definition of viscosity solution expressed in terms of 
the sub- and superdifferential of the function (see~\cite{CEL}), 
and because  of Theorem~\ref{Pro-Semi}-$(iv)$,
one immediately see that every $C^1$ solution of~\eqref{HJ} is also a viscosity solution
of~\eqref{HJ}. On the other hand, if $u$ is a viscosity solution of~\eqref{HJ}, then $u$ satisfies the equation
at every point of differentiability. Moreover, by the definition of reachable
gradient, it follows that there holds
\begin{equation}
\label{hj-reach-grad}
p_t+H(p_x)=0\qquad\forall~(p_t,p_x)\in D^*u(t,x)\,,
\end{equation}
at any $(t,x)\in[0,+\infty[\times\mathbb{R}^N$.
\end{remark}
It is well-known that a viscosity solution $u$ of~\eqref{HJ} is locally semiconcave
(see, for instance, \cite[Theorem~5.3.8]{CS}).
Relying on the properties of the semiconcave functions recalled in the previous section,
one can prove further regularity  for viscosity solutions which will be useful in the paper.
\begin{proposition}
\label{co:smoothness}
Let $u:[0,T]\times \R^N$ be a viscosity solution of~\eqref{HJ} 
and assume that 
$u(t,\cdot )$ is both semiconcave and semiconvex in $\R^N$ for all $t\in \,]0,T]$. 
Then $u$ is a continuously differentiable 
classical solution of~\eqref{HJ} on~$]0,T]\times \R^N$.
\end{proposition}
In other words,  smoothness in the pair $(t,x)$ follows from smoothness in the second variable.  We give a proof for the reader's convenience.
\begin{proof}
Since a viscosity solution is locally semiconcave, relying on property $(i)$ of Theorem~\ref{pro-semiconc}
and properties $(ii)$, $(iv)$ of Theorem~\ref{Pro-Semi}
it follows that, in order to show that $u$ is everywhere continuously
differentiable, it is sufficient to prove
that the superdifferential $D^+u(t,x)$ is a singleton for all $(t,x)\in \,]0,T]\times\R^N$. 
In turn, the differentiability of $u$ implies that the equation~\eqref{HJ} is pointwise satisfied in the
classical sense by Remark~\ref{visos-sol-differentiab}.
Then, fix $(t_0,x_0)\in \,]0,T]\times\R^N$ and observe that, by Remark~\ref{semiconc-semiconc},
$u(t_0,\cdot )$ is  differentiable at $x_0$
since it is both semiconcave and semiconvex in $\R^N$.
Therefore, by property~$(iv)$ of Theorem~\ref{Pro-Semi}, 
the superdifferential $D_x^+u(t_0,x_0)$ of $u(t_0,\cdot)$ at $x_0$ is the singleton $\{\nabla_{\!x} u(t_0,x_0)\}$.
On the other hand,
invoking a well-known property of the superdifferential 
(see, for instance, \cite[Lemma~3.3.16]{CS}) we deduce that $\Pi_x\, D^+u(t_0,x_0)=D_x^+u(t_0,x_0)$, where $\Pi_x$ 
denotes the projection of $\R\times\R^N$ onto $\R^N$ defined by $\Pi_x(t,x)=x$.
Hence, recalling property~$(i)$ of Theorem~\ref{Pro-Semi}, 
we get
$
D^+u(t_0,x_0)=[\tau_-,\tau_+]\times\{\nabla_{\!x} u(t_0,x_0)\}
$
for some $\tau_\pm\in\R$, with $\tau_-\leqslant \tau_+$. This implies that  $(\tau_\pm, \nabla_{\!x} u(t_0,x_0))\in D^*u(t_0,x_0)$
by property~$(iii)$ of Theorem~\ref{Pro-Semi}. So,
applying~\eqref{hj-reach-grad}, we find
$
\tau_-+H\big( \nabla_{\!x} u(t_0,x_0))\big)=0=\tau_++\big( \nabla_{\!x} u(t_0,x_0))\big)
$,
which in turn yields $\tau_-=\tau_+$, showing that $D^+u(t_0,x_0)$ is a singleton as desired.
\end{proof}
\noindent
Further analysis  shows that, with the same hypotheses of Proposition~\ref{co:smoothness}, 
the viscosity solutions of \eqref{HJ} have a locally Lipschitz gradient in $(t,x)$.
\bigskip

Under  assumption \textbf{(H1)}$'$, 
the viscosity solution of the Hamilton-Jacobi equation (\ref{HJ}) with initial data $u(0,\cdot)=u_{0}\in \Lip(\mathbb{R}^N)$
can be represented as the value function
of a classical problem in calculus of variation, which admits the Hopf-Lax representation formula 
\begin{equation}\label{Hopf}
u(t,x)=\min_{y\in\mathbb{R}^N}\bigg\lbrace{t\cdot H^*\Big(\frac{x-y}{t}\Big)+u_0(y)\bigg\rbrace},
\qquad t>0, \ \ x\in \mathbb{R}^N\,,
\end{equation} 
where $H^*$ denotes the Legendre transform of $H$, defined by
\begin{equation}\label{Legendre-L}
H^*(q):=\max_{p\in\mathbb{R}^N}\big\{p\cdot q-H(p)\big\} \qquad q\in\mathbb{R}^{N}.
\end{equation}
The Legendre transform inherits the properties of $H$ (cfr.~\cite[Appendix A.2]{CS}). In particular,   assumption  \textbf{(H1)}$'$
implies that $H^*\in C^2(\mathbb{R}^N)$ and  $H^*$ is a uniformly convex  coercive
map, 
i.e.,
\begin{equation}
\label{suplin-legendre}
\lim_{|p|\rightarrow\infty}\frac{H^*(p)}{|p|}=+\infty.
\end{equation}
Moreover,  $\nabla H^*$ is a $C^1$ diffeomorphisms on $\R^N$ as $\nabla H$,
and one has
\begin{equation}
\label{HH*-grad}
\qquad
(\nabla H^*)^{-1}(p)=\nabla H(p),
\qquad \quad
D^2 H^*(p) = \Big(D^2H (\nabla H^*(p))\Big)^{\!-1}
\qquad\quad\forall~p\in\R^N\,.
\end{equation}
On the other hand, the lower bound  bound on the Hessian matrix $D^2H$ given in~{\bf (H1)} implies
\begin{equation}
\label{low-bound-hess-legendre}
D^2H^*\leq\frac{1}{\alpha}\cdot \mathbb{I}_{N},
\end{equation}
while, by virtue of~{\bf (H2)}, we have
\begin{equation}
\label{HH*0}
H^*(0)=-H(0).
\end{equation}

The main properties of viscosity solutions defined by the Hopf-Lax formula
of interest to this paper are recalled below
(cfr.~\cite[Section 1.2, Section 6.4]{CS}, \cite[Section 3.3]{E}). 
\begin{proposition}\label{Pro}
Let $u$ be the viscosity solution of \eqref{HJ}  on $[0,+\infty[\,\times\R^N$, with initial data $u_0\in\Lip(\mathbb{R}^N)$,
defined by~\eqref{Hopf}. Then the following holds true.
\begin{enumerate}
\item[$(i)$] \textbf{Functional identity}: for all $x\in\mathbb{R}^N$ and $0\leqslant s<t$,
it holds
\[
u(t,x)=\min_{y\in\mathbb{R}^N}\Big\lbrace{u(s,y)+(t-s)\cdot H^*\Big(\frac{x-y}{t-s}\Big)\Big\rbrace}\,.
\]
\item[$(ii)$] \textbf{Differentiability of $u$ and uniqueness}: 
for all $x\in\mathbb{R}^N$ and $t> 0$, any minimizer $y_x$ of~\eqref{Hopf}
satisfies $y_x\in \{x-t\cdot \nabla H(p)\, | \, p\in D^*_{\!x}(u(t,x)\}$, where $D^*_{\!x}(u(t,x)$ denotes the reachable gradient of
$u(t,\cdot)$ at $x$.
Moreover, \eqref{Hopf} admits a unique minimizer $y_x$ if and only if $u(t,\cdot)$ is differentiable at $x$. In this case we have that $y_x=x-t\cdot \nabla H\big(\nabla_{\!x}(u(t,x)\big)$.
\item[$(iii)$] \textbf{Dynamic programming principle}: let \, $t>s>0$, $x\in\mathbb{R}^N$,  assume that $y$ is a minimizer for \eqref{Hopf}, and define $z=\frac{s}{t}x+(1-\frac{s}{t})y$. Then $y$ is the unique minimizer over~$\R^N$ of
\begin{equation*}
w\mapsto s\cdot H^*\Big(\frac{z-w}{s}\Big)+u_0(w)\qquad(w\in\R^n)\,.
\end{equation*}
\end{enumerate} 
\end{proposition}
By the above observations and because of Proposition~\ref{Pro}-(i),
the family of nonlinear operators
$$S_t:\Lip(\mathbb{R}^N)\rightarrow\Lip(\mathbb{R}^N), 
\qquad\quad u_0\mapsto S_t u_0, \ \ t\geqslant 0,$$
defined by
\begin{equation}
\label{hopf-lax-smgr}
\begin{cases}
\hspace{.cm}
S_t u_0(x):=\min_{y\in\mathbb{R}^N}\Big\lbrace{t\cdot H^*\big(\frac{x-y}{t}\big)+u_0(y)\Big\rbrace}
&t>0\,,\;x\in\R^N,
\\S_0 u_0(x) :=u(x)&x\in\R^N,
\end{cases}
\end{equation}
enjoy the following properties:
\begin{itemize}
\item[(i)] for every  $u_0\in\Lip(\mathbb{R}^N)$, $u(t,x):=S_t u_0(x)$ provides the unique viscosity solution of the 
Cauchy problem~\eqref{HJ}-\eqref{in-data};
\item[(ii)] (semigroup property)
$$S_{t+s} u_0 = S_t\, S_s u_0
\,,\quad\forall t,s\geqslant 0\,,\;\forall u_0\in\Lip(\mathbb{R}^N);$$
\item[(iii)] for every constant $c\in\R$ we have that
\begin{equation}\label{eq:support}
S_t(u_0+c)=S_t u_0+c\,,\qquad\forall u_0\in\Lip(\R^N)\,,\;\forall t\geqslant 0\,.
\end{equation}
\end{itemize}
It's a well-known fact that, for every fixed $t\geq 0$, the map $S_t$ is continuos with respect to the
topology of uniform convergence on compact sets.
We next provide a proof of the continuity of such a map also in the case where the space 
$\Lip(\mathbb{R}^N)$ is endowed with the $\mathbf{W}^{1,1}_{loc}$-topology
and $S_t$ is restricted to  sets of functions with uniform Lipschitz constant.
Namely, the following holds.
\begin{proposition}
\label{St-W11-continuity}
Let $u, u^\nu\in \Lip(\mathbb{R}^N)$  $(\nu \in \N)$ be such that
%
\begin{gather}
\label{unu-bound}
\qquad
\Lip[u^\nu]
\leqslant M\qquad\forall~\nu\,,
\qquad\text{for some} \ \ M>0\,,
%
\\
\noalign{\smallskip}
\label{unu-conv}
u^\nu \ \ \underset{\nu\to\infty}\longrightarrow \ \  u 
\qquad\text{in}\qquad \mathbf{W}^{1,1}_{loc}(\R^N)\,.
\end{gather}
Then, for every fixed $t\geq 0$, one has
\begin{equation}
\label{conv-semigr}
S_t u^\nu\ \ \underset{\nu\to\infty}\longrightarrow \ \  S_t u 
\qquad\text{in}\qquad \mathbf{W}^{1,1}_{loc}(\R^N)\,.
\end{equation}
\end{proposition}
\begin{proof}
In order to establish the proposition it will be sufficient to show that, 
for every given bounded domain $\Omega\subset\R^N$, and for any fixed $t\geq 0$, 
there holds
\begin{equation}
\label{conv-semigr-2}
S_t u^\nu \ \ \underset{\nu\to\infty}\longrightarrow \ \  S_t u 
\qquad\text{in}\qquad \mathbf{W}^{1,1}(\Omega)\,.
\end{equation}
Observe that, because of~\eqref{unu-bound}, and relying on the
a-priori bound on the gradient of the solution to~\eqref{HJ} provided by 
Lemma~\ref{ProV} in the next section, we have 
$\Lip[S_t u^\nu] \leqslant M$ for all $\nu$.
In turn, this implies $|p|\leq M$ for all $p\in D_x^*\,S_t u^\nu(x)$, $x\in\R^N$, and for any $\nu$.
Thus, invoking Proposition~\ref{Pro}-$(ii)$ we deduce that, for all $x\in\Omega$
and for any minimizer $y_x^\nu$ of~\eqref{Hopf}, with $u^\nu$ in place of $u_0$, one has
\begin{equation}
\label{o-prime-def}
y_x^\nu\in \Omega':=\big\{x\in\R^N\, | \, d(x,\Omega)\leq t\cdot \sup_{|p|\leqslant M}|\nabla H(p)|\big\}
\qquad\forall~\nu\,.
\end{equation}
Next, notice that because of~\eqref{unu-bound}, \eqref{unu-conv}, 
letting $\overline x\in\Omega$ be a point such that $u(\overline x)= \lim_{\nu\to\infty} u^\nu(\overline x)$, 
we have
\begin{equation}
\label{equicont}
\begin{aligned}
\big|u^\nu(x)\big|\leq  &\sup_\nu \big|u^\nu(\overline x)\big|+ M \cdot \mathrm{diam}(\Omega') <+\infty
\qquad\forall~x\in\overline{\Omega'}\,,
\\
\noalign{\smallskip}
&\big|u^\nu(x)-u^\nu(y)\big|\leq M \cdot \big|x-y\big|\qquad\quad\forall~x,y\in\overline{\Omega'}\,,
\end{aligned}
\qquad\quad\forall~\nu\,.
\end{equation}
Therefore, by a standard argument based on~\eqref{unu-conv} and the Ascoli-Arzel\`a compactness theorem,
we deduce that 
\begin{equation}
\label{univ-donv-unu}
u^\nu \ \ \underset{\nu\to\infty}\longrightarrow \ \ u\qquad\text{uniformly on}\quad\Omega'\,.
\end{equation}
Repeating the same reasoning for every bounded domain of $\R^N$ it 
follows that $u^\nu$ converges to~$u$
(uniformly on compact sets) on the whole space $\R^N$ and that $\Lip[u]\leq M$.
Hence, for all $x\in\Omega$
and for any minimizer $y_x$ of~\eqref{Hopf}, with $u$ in place of $u_0$, one has $y_x\in\Omega'$.
In turn,  together with~\eqref{o-prime-def} and \eqref{univ-donv-unu}, this fact
implies that 
\begin{equation}
\label{min-conv}
\qquad\quad
\min_{y\in\mathbb{R}^N}\Big\lbrace{t\cdot H^*\big(\tfrac{x-y}{t}\big)+u^\nu(y)\Big\rbrace}
\ \ \underset{\nu\to\infty}\longrightarrow \ \
\min_{y\in\mathbb{R}^N}\Big\lbrace{t\cdot H^*\big(\tfrac{x-y}{t}\big)+u(y)\Big\rbrace}
\qquad\text{uniformly on} \ \ \Omega\,,
\end{equation}
which, by virtue of definition~\eqref{hopf-lax-smgr}, yields
\begin{equation}
\label{conv-Stu}
S_t u^\nu(x) \ \ \underset{\nu\to\infty}\longrightarrow \ \ S_t u(x)\qquad\text{uniformly on} \ \ \Omega\,.
\end{equation}
As a consequence, we deduce that
\begin{equation}
\label{conv-semigr-L1}
S_t u^\nu \ \ \underset{\nu\to\infty}\longrightarrow \ \  S_t u 
\qquad\text{in}\qquad \mathbf{L}^{1}(\Omega)\,.
\end{equation}
On the other hand, observe that by Proposition~\ref{Pro}-$(ii)$
it follows that $S_t u^\nu$, $S_t u$ are differentiable almost everywhere
in $\Omega$, and  there holds
\begin{equation}
\label{minimiz-unu}
\begin{aligned}
y^\nu_x &= x - t\cdot \nabla H\big(\nabla S_t u^\nu (x)\big)
\quad\forall~\nu\,,
\\
\noalign{\smallskip}
y_x &= x - t\cdot \nabla H\big(\nabla S_t u (x)\big),
\end{aligned}
\qquad\quad\text{for \ a.e.}\ \ x\in\Omega\,,
\end{equation}
where $y^\nu_x,\, y_x$ denotes the unique minimizer 
of~\eqref{Hopf}, with $u^\nu$ and $u$ in place of $u_0$, respectively.
Moreover, because of the uniqueness of such minimizers of~\eqref{Hopf},
and by virtue of the convergence~\eqref{min-conv}, \eqref{conv-Stu}, 
we deduce that $\{y^\nu_x\}_\nu$ converges to $y_x$ for almost every $x\in \Omega$.
Thus, relying on~\eqref{minimiz-unu}, and recalling that $\nabla H$ is a diffeomorphism on $\R^N$, we infer that
\begin{equation}
\label{grad-Stunu-conv}
\nabla S_t u^\nu (x)
\ \ \underset{\nu\to\infty}\longrightarrow \ \ 
\nabla S_t u (x)
\qquad\text{for \ a.e.}\ \ x\in\Omega\,.
\end{equation}
On the other hand,  these unique minimizers satisfy $y^\nu_x \in\Omega'$ for all $\nu$
and for almost every $x\in\Omega$, 
so that
one has
\begin{equation}
\label{minimiser-bounbd}
\bigg|\frac{x-y^\nu_x}{t}\bigg|\leqslant \sup_{|p|\leqslant M}|\nabla H(p)|
\qquad\forall~\nu\,,\qquad\text{for \ a.e.}\ \ x\in\Omega\,.
\end{equation}
Thus, because of~\eqref{minimiz-unu}, \eqref{minimiser-bounbd},
we derive a uniform $\mathbf{L}^\infty$
bound on $\nabla S_t u^\nu$, $\nu\in \N$, over $\Omega$, which, together with~\eqref{grad-Stunu-conv}, implies
\begin{equation}
\label{conv-semigr-nabla-L1}
\nabla S_t u^\nu \ \ \underset{\nu\to\infty}\longrightarrow \ \  \nabla S_t u 
\qquad\text{in}\qquad \mathbf{L}^{1}(\Omega)\,.
\end{equation}
Then, from \eqref{conv-semigr-L1}, \eqref{conv-semigr-nabla-L1}
we recover~\eqref{conv-semigr-2}, concluding the proof of the proposition.
\end{proof}
\bigskip

\section{Upper estimates}
\label{sec:up-ext}
\subsection{A-priori bounds on the Hopf-Lax semigroup}
Let $H:\R^N\to \R$ be  a function satisfying  the assumptions {\bf (H1)-(H2)}.
We  collect here  some a-priori bounds  on the
semiconcavity costant and on the gradient of the solutions to~\eqref{HJ}
and we  establish an a-priori bound  on the size of their support.
Namely, given $L,M>0$,  consider the set of initial data introduced in~\eqref{Cclass}:
\begin{equation*}
\mathcal C_{[L,M]}=\Big\lbrace{u_0\in \Lip(\mathbb{R}^N)\ \big|\ \mathrm{supp}(u_0)\subset [-L,L]^N\,,\; 
\Lip[u_0]\leqslant M\Big\rbrace}.
\end{equation*}
The image of $\mathcal C_{[L,M]}$ through the Hopf-Lax 
semigroup map $S_T$ defined in~\eqref{hopf-lax-smgr} enjoy the properties stated in the following
\begin{lemma}\label{ProV} 
For any $L, M, T>0$ and for every $u_0\in \mathcal C_{[L,M]}$, 
the following properties hold true:
\begin{enumerate}
\item[(i)] $S_T u_0$ is semiconcave in $\R^N$ with constant $\frac{1}{\alpha T}$;
\item[(ii)] $\Lip[S_T u_0]
\leqslant M$;
\item[(iii)] 
$\mathrm{supp}\big(S_T u_0+T\cdot H(0)\big)\subset \big[\!-\!l_{[L,M,T]},\,l_{[L,M,T]}\big]^N$,
where $l_{[L,M,T]}$ is the constant defined in~\eqref{LT}.
%
\end{enumerate}
\end{lemma}
\smallskip
\begin{proof}
Under the assumption {\bf (H1)} and recalling~\eqref{low-bound-hess-legendre},
property $(i)$ is well-known (see~\cite[Corollary~1.6.2]{CS}),
while $\Lip[u_0]\leqslant M$
and an application of  \cite[Theorem~1.3.2]{CS} implies
\[
|S_T u_0 (y)-S_T u_0 (x)|\leqslant  M\,|y-x|,\quad\forall x,y\in\mathbb{R}^N,
\]
which yields $(ii)$.
\medskip

Concerning a proof of $(iii)$, by
the Lipschitz continuity of $S_T u_0$
it will be sufficient to show that at every point
$x\in\mathbb{R}^N\backslash [-l_{[L,M,T]},\,l_{[L,M,T]}]^N$ where $S_T u_0$ is differentiable there holds
\begin{equation}\label{eq:constant}
S_T u_0 (x)=-T\cdot H(0).
\end{equation}
Indeed, recalling~\eqref{hopf-lax-smgr} and invoking Proposition~\ref{Pro}-$(ii)$, we find  that
at every such point $x$ one has
\begin{equation}
\label{hopfu0}
S_T u_0 (x)=T\cdot H^*\left(\nabla H\big(\nabla S_T u_0 (x)\big)\right)+u_0(y_x)\,,
\end{equation}
where 
\begin{equation}\label{CL}
y_x=x-T\cdot \nabla H\big(\nabla S_T u_0 (x)\big)\,.
\end{equation}
Observe now that, relying on the property $(ii)$ above established and recalling (\ref{LT}), we deduce that
$y_x\in \mathbb{R}^N\backslash[-L,L]^N$ for all $x\in\mathbb{R}^N\backslash [-l_{[L,M,T]},l_{[L,M,T]}]^N$.
This, in turn, implies 
\begin{equation}
\label{u0-cond}
u_0(y_x)=0,\qquad\quad\nabla u_0(y_x)=0,
\end{equation}
because
$\mathrm{supp}(u_0)\subset [-L,L]^N$ by the definition~\eqref{Cclass} of the set $\mathcal{C}_{[L,M]}$.
Moreover, since by Proposition~\ref{Pro}-$(ii)$ $y_x$ is a minimum of
$$
y\mapsto T\cdot H^*\Big(\frac{x-y}{T}\Big)+u_0(y)
$$
over $\mathbb{R}^N$,
it follows that  
$-\nabla H^*\big(\frac{x-y_x}{T}\big)=\nabla u_0(y_x)$.
Hence, relying on~\eqref{CL}, \eqref{u0-cond}, we deduce that 
\begin{equation}
\label{u0-cond2}
\nabla H^*\big(\nabla H\big(\nabla S_T u_0 (x)\big)\big)=\nabla H^*\Big(\frac{x-y_x}{T}\Big)=0.
\end{equation}
Thus, by virtue of~\eqref{hopfu0}, \eqref{u0-cond}, \eqref{u0-cond2}, we conclude that
$\nabla S_T(u) (x)=0$  at every point
$x\in\mathbb{R}^N\backslash [-l_{[L,M,T]},\,l_{[L,M,T]}]^N$ where $S_T u_0$ is differentiable.
 This, in turn, by the assumption {\bf (H2)} and because of~\eqref{hopfu0}, \eqref{u0-cond},
 implies that at every such point $x$ there holds
\begin{equation}
\label{u0-cond3}
S_T u_0 (x)=T\cdot H^*(\nabla H(0)) = T \cdot H^*(0).
\end{equation}
Finally, recalling~\eqref{HH*0}, we recover 
\eqref{eq:constant} from~\eqref{u0-cond3}, thus completing the proof of~$(iii)$.
\end{proof}
\smallskip

\begin{remark}\rm
Property $(iii)$ of Lemma~\ref{ProV}  implies that, for every $u_0\in \mathcal C_{[L,M]}$, the domain \linebreak
$\mathrm{supp}\big(S_t u_0+t\cdot H(0)\big)$
where $S_t u_0$ 
differs from the constant in  space solution with zero initial data propagates at a finite speed 
as illustrated in Figure~1 below.

\setlength{\unitlength}{1mm}
\begin{picture}(150,55)(-30,-5)
\linethickness{1pt}
\put (2,15){\vector(1,0){85}} 
\put (40,5){\vector(0,1){40}}
{\Cbe \qbezier(30,15)(40,30)(50,15)}
{\Cr\qbezier(20,10)(25,11)(30.2,25)}
{\Cr\qbezier(29,25)(32,35)(37.5,38)}
{\Cr\qbezier(36.5,38)(43,35)(47.2,25)}
{\Cr\qbezier(46,25)(51,11)(60,10)}
\put (60,10){\vector(1,0){20}} 
\put(16,10){\line(-1,0){20}}
\put(27,20){\makebox(0,0)[b]{$u_0$}}
\put(37,42){\makebox(0,0)[b]{$u$}}
\put (49,15){\Cbe\vector(1,0){3}} 
\put (60,15){\vector(0,-1){5}} 
\put(24,30){\makebox(0,0)[b]{$S_T u_0$}}
\put(80,18){\makebox(0,0)[b]{$x$}}
\put(37,10){\makebox(0,0)[b]{$0$}}
\put(46,17){\makebox(0,0)[b]{$L$}}
\put(63,17){\makebox(0,0)[b]{$L_T$}}
\put(5,5){\makebox(0,0)[b]{$-T\cdot H(0)$}}
\put(40,-3){\makebox(0,10)[b]{${\rm Figure \; 1:
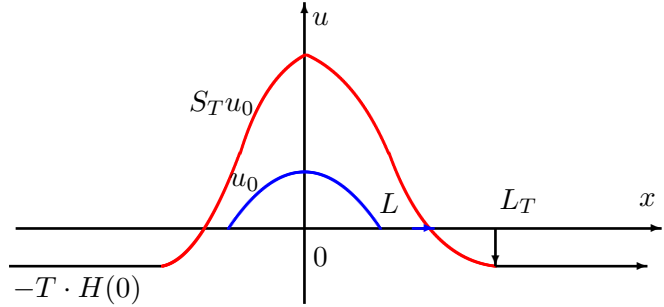
evolution \;of \;the\; support\; of\; the\; gradient\; under} \;S_T .$}}
\end{picture}
\\
Having in mind the a-priori bound established in~\cite{DLG} for the support of solutions to scalar conservation laws with convex
flux, one may wonder whether is it
possible to derive a sharper estimate on the size
of such a domain. 
In fact, if we consider a class of initial data 
\begin{equation}
\label{CLmM}
\mathcal C_{[L,m,M]}:=\Big\lbrace{u_0\in \Lip(\mathbb{R}^N)\ \big|\ \mathrm{supp}(u_0)\subset [-L,L]^N,
\
\|\nabla u_0\|_{\mathbf{L}^1(\mathbb{R}^N)}\leqslant m\,, 
\Lip[u_0]\leqslant M\Big\rbrace},
\end{equation}
one may look for establishing an estimate as 
\begin{equation}
\label{supp-m-bound}
\big|\mathrm{supp}\big(S_T u_0+T\cdot H(0)\big)\big|\leq 
\bigg(2L+\sup_{|p|\leqslant M}\|D^2H(p)\|\cdot 4\sqrt{\frac{m\, T}{\alpha}\,}\,\bigg)^{\!\!N}
\end{equation}
relying on property $(i)$ of Lemma~\ref{ProV} and property $(iii)$ of Theorem~\ref{pro-semiconc}. 
However, a key point in the proof of an estimate of this
type for the support of solutions to scalar conservation laws is the fact that, for such equations,
the $\mathbf{L}^1$-norm of the solution is non increasing in time as a consequence of the $\mathbf{L}^1$ contractivity
of the semigroup map $S_t$. This property continues to hold for the gradient of solutions to
Hamilton-Jacobi equations in one space dimension, but it is no more true in general when the space
dimension is greater than one. In fact in this case, as observed in the introduction, the gradient of
a solution of an Hamilton-Jacobi equation turns out to be a solution of an hyperbolic system of conservation laws
and it is well-known that for general hyperbolic systems of conservation laws no metric is contractive~\cite{T}.
As a consequence, one can easily convince himself that a bound as~\eqref{supp-m-bound} doesn't hold
for Hamilton-Jacobi equations in several space variables.
This is the main reason for which we limit ourself to analyze in this paper the image through
the Hopf-Lax semigroup~$S_t$ of sets of  initial data of the form~\eqref{Cclass}
and we don't consider  sets of the form~\eqref{CLmM}.
\end{remark}
\medskip

Given any $L,M,K>0$, consider now the class of functions
\begin{equation}
\label{semiconc-def}
\mathcal{SC}_{[L,M,K]} :=
\Big\{
u\in \mathcal C_{[L,M]}\, \big| \, u \ \text{is semiconcave with semiconcavity constant} \ K\Big\},
\end{equation}
where  $\mathcal C_{[L,M]}$ denotes the set in~\eqref{Cclass}.
Then, applying Lemma~\ref{ProV}, we immediately obtain the following.
\begin{proposition}\label{Com-entro} Let $H:\R^N\to \R$ be  a function satisfying  the assumptions~{\bf (H1)-(H2)}
and $\{S_t : \Lip (\mathbb{R}^N) \to \Lip (\mathbb{R}^N)\}_{t\geqslant 0}$ be the semigroup of viscosity solutions generated by~\eqref{HJ}. 
Then, given  any $L,M,T>0$, there holds
\begin{equation}\label{Inclusion1}
S_{T}\big(\mathcal C_{[L,M]}\big)+T\cdot H(0)\subset \mathcal{SC}_{\big[l_{[L,M,T]},\,M,\,\frac{1}{\alpha T}\big]},
\end{equation}
where $l_{[L,M,T]}$ is given by~\eqref{LT} and $\alpha$ is
the constant  in {\bf (H1)}.
\end{proposition}

\subsection{An upper bound on the $\varepsilon$-entropy for semiconcave functions}
Towards a derivation of an upper bound on the $\varepsilon$-entropy in $\mathbf{W}^{1,1}$ for the class of semiconcave functions
introduced in~\eqref{semiconc-def}, in view of  Proposition~\ref{prop2}-$(ii)$ we shall first establish
an upper bound on the $\varepsilon$-entropy in $\mathbf{L}^1$ for a class of monotone multifunctions 
with uniformly bounded total variation defined on a cube of $\R^N$.
As observed in Section~\ref{subsec:semiconc-monotone}, any monotone multifunction is
almost everywhere univalued in the interior of its domain, and can be regarded as a
function of bounded variation on this set.
Hence, set $I_L:=  \,]-L,L[$,
$I_M:= \,]-M,M[$, and
consider the class of monotone multifunction
\begin{equation}
\label{monotone-class}
\mathcal{F}_{[L,M,C]} := \Big\lbrace{F:I_L^N\rightarrow 2^{I_M^N}\ \big|\ \text{dom}(F)=I_L^N,
\ \ F \ \text{is decreasing}, \  \ |DF|(I_L^N)
\leqslant C\Big\rbrace},
\end{equation}
where $ |DF|$ denotes the total variation of the matrix-valued Radon measure DF.
With a slight abuse of notation, we shall regard $\mathcal{F}_{[L,M,C]}$
as a subset of $\mathbf{L}^1\big(I_L^N, I_M^N\big)$ consisting of all functions in $\mathbf{L}^1\big(I_L^N, I_M^N\big)$
that coincide almost everywhere  with an element of the set defined in~\eqref{monotone-class}.
\begin{proposition}\label{Entropy1}
Given $L,M>0$, 
for  any $\varepsilon >0$ sufficiently small
there holds
\begin{equation}
\label{monotone-upbound}
\mathcal{H}_{\epsilon}\Big(\mathcal{F}_{[L,M,C]}\ |\ \mathbf{L}^1\big(I_L^N, I_M^N\big)\Big)\leqslant \gamma_{_{[L,M,C,N]}}\cdot\frac{1}{\varepsilon^N},
\end{equation}
where 
\begin{equation}
\label{Glmcn-def}
\gamma_{_{[L,M,C,N]}} := 2^{(N^2+N+1)}\,N^{\big(\frac{N}{2}+1\big)}\cdot\big(M L^N+LC\big)^N.
\end{equation}
\end{proposition}
\begin{proof}
\quad\\
{\bf 1.} Towards a proof of~\eqref{monotone-upbound}, we shall associate to any 
function $F\in \mathcal{F}_{[L,M,C]}$, 
a piecewise constant function $\widetilde F\in \mathbf{L}^{\infty}\big(I_L^N, I_M^N\big)$
that takes values in a discrete subset of~$I_M^N$
and has the property that every $i$-th component is (almost everywhere) monotone decreasing in the $i$-th variable.  
Namely, given any fixed $n\in\mathbb{N}$, we divide $]\!-\!L,L[^N$ into $n^N$ cubes with sides 
of length~$\frac{2L}{n}$ as follows. For every multiindex $\iota=(\iota_1,...,\iota_N)\in\lbrace{0,...,n-1\rbrace}^N$
we define the cube
%
\begin{equation*}
\label{cube-def}
\square_{\iota}
:= 
\,\bigg]\!-\!L+\frac{\iota_1}{n}2L,\, -L+\frac{\iota_1+1}{n}2L\bigg[\,\times \cdots\times\,
\bigg]\!-\!L+\frac{\iota_N}{n}{2L},\, -L+\frac{\iota_N+1}{n}2L\bigg[\,,
\end{equation*}
so that one has
\begin{equation}
\label{cube-decomp}
{\overline I}_L^N=\bigcup_{\iota\in\lbrace{0,...,n-1\rbrace}^N} \overset{\text{---}}{\square}_{\iota},
\end{equation}
where ${\overline I}_L=[-L, L]$, and $\overset{\text{---}}{\Box}_{\iota}$ denotes the closure of $\Box_{\iota}$.
Then, given any $F\in\mathcal{F}_{[L,M,C]}$, for every $\iota\in \lbrace{0,...,n-1\rbrace}^N$, let
\begin{equation}
\label{F-average-def}
\overline{F}_{\iota}:= 
\frac{1}{\mathrm{Vol}(\Box_{\,\iota})}\cdot\int_{\square_{\,\iota}}F(x)dx
\end{equation}
be the average of $F$ over $\Box_{\iota}$. 
Observe that $\overline{F}_{\iota}=(\overline{F}^1_{\iota},...,\overline{F}^N_{\iota})\in I_M^N$
since $F$ takes values in $I_M^N$.
Next,
consider the subdivision of $[-M,M]$ into the $n$ intervals
\[
[-M,M]=\Big[\!-\!M,\, \!-\!M+{\textstyle\frac{2M}{n}}\Big[\cup\cdots\cup\Big[\!-\!M+{\textstyle\frac{2(n-2)}{n}}M,\, \!-\!M+
{\textstyle\frac{2(n-1)}{n}}M\Big[\cup
\Big[\!-\!M+{\textstyle\frac{2(n-1)}{n}}M,M\Big],
\]
and
define the vector $\widetilde{F}_{\iota}=(\widetilde{F}^1_{\iota},...,\widetilde{F}^N_{\iota})\in I_M^N$
by setting, for each $i\in\{1,\dots,N\}$,
\begin{equation}
\label{Ftilde-def1}
\widetilde{F}_{\iota}^i:= 
\begin{cases}
\!-\!M\!+\big(k+\frac{1}{2}\big)\cdot\frac{2M}{n},
\quad
&\text{if}\qquad\overline{F}^i_{\iota}\in \Big[\!-\!M\!+2k\cdot\frac{M}{n},\, \!-M\!+2(k+1)\cdot\frac{M}{n}\Big[,
\qquad k \leqslant n-2,
\\
\noalign{\smallskip}
\!-\!M\!+\big(n-\frac{1}{2}\big)\cdot\frac{2M}{n},
\quad
&\text{if}\qquad
\overline{F}^i_{\iota}\in \Big[\!-\!M\!+(n-1)\cdot\frac{M}{n}, M\Big]\,.
\end{cases}
\end{equation}
Then, let $\widetilde F\in \mathbf{L}^{\infty}\big(I_L^N, I_M^N\big)$
be the function defined almost everywhere
by
\begin{equation}
\label{Ftilde-def2}
\qquad
\widetilde{F}(x):=\widetilde{F}_{\iota}
\qquad\mathrm{if}\qquad x\in \square_{\iota}\,, \
\iota\in\lbrace{0,...,n-1\rbrace}^{N}\,,
\end{equation}
with $\widetilde{F}_{\iota}=(\widetilde{F}^1_{\iota},...,\widetilde{F}^N_{\iota})$ as in~\eqref{Ftilde-def1}.
By construction, we have
\begin{equation}
\label{ftilde-range}
\widetilde{F}(x)\in J_{M,n}:=
\bigg\lbrace{\!-\!M\!+\Big(k+\frac{1}{2}\Big)\cdot\frac{2M}{n}\ \bigg|\
k= 0, \dots, n-1
\bigg\rbrace}\qquad\forall~x\in\bigcup_{\iota}\, {\square}_{\iota}\,.
\end{equation}
We claim that $F$ enjoys the following two properties:
\begin{itemize}
\item[$(i)$] For every $\iota=(\iota_1,...,\iota_N)\in\lbrace{0,...,n-1\rbrace}^{N}$ 
there holds
\begin{equation}
\label{increasing1}
\iota_i<n-1\quad
\Longrightarrow\quad
\widetilde{F}^{i}_{\iota}\geqslant
\widetilde{F}^{i}_{\iota+e_i}\,,  
\end{equation}
where $e_i$ denotes the $i$-th element of the canonical basis of $\R^N$.
\item[$(ii)$]
\begin{equation}
\label{est2}
\big\|\widetilde{F}-F\big\|_{\mathbf{L}^1(I_L^N, I_M^N)}\leqslant \gamma^1_{_{[L,M,C,N]}}\cdot\frac{1}{n}
\end{equation}
where 
\begin{equation}
\label{gamma1-def}
 \gamma^1_{_{[L,M,C,N]}} := {\sqrt{N}}\Big(M(2L)^N+LC\Big).
\end{equation}
\end{itemize}
In fact, given any $\iota=(\iota_1,...,\iota_N)\in\lbrace{0,...,n-1\rbrace}^{N}$
and $\iota_i<n-1$,  by definition~\eqref{F-average-def} 
and since $F$ is monotone decreasing
we find
\begin{equation}
\label{increasing}
\begin{aligned}
\overline{F}^i_{\iota+e_i}-\overline{F}^i_{\iota}&=
\big\langle \overline{F}_{\iota+e_i}-\overline{F}_{\iota}, e_i\big\rangle
\\
\noalign{\medskip}
&=\Bigg\langle
\frac{n^N}{(2L)^N}\cdot\int_{\square_{\iota}}\bigg(F\Big(x+{\textstyle\frac{2L}{n}}e_i\Big)-F(x)\bigg)dx,\, e_i
\Bigg\rangle
\\
\noalign{\medskip}
&=\frac{n^{N+1}}{(2L)^{N+1}}\cdot\int_{\square_{\iota}}\bigg\langle F\Big(x+{\textstyle\frac{2L}{n}}e_i\Big)-F(x),\, x+{\textstyle\frac{2L}{n}}e_i-x\bigg\rangle dx
\\
\noalign{\medskip}
&\leq 0.
\end{aligned}
\end{equation}
 By definition~\eqref{Ftilde-def1},  \eqref{increasing}  implies, in turn, $\widetilde{F}^{i}_{\iota+e_i}\leqslant
\widetilde{F}^{i}_{\iota}$, thus proving~\eqref{increasing1}.
Concerning~\eqref{est2}, observe first that,   
by definition~\eqref{F-average-def} 
and  relying on the Poincar\'e inequality for BV functions stated in Theorem~\ref{poincare},
for any $\iota=(\iota_1,...,\iota_N)\in\lbrace{0,...,n-1\rbrace}^{N}$
we derive
\begin{equation}
\label{Poincare}
\big\|F-\overline{F}_{\iota}\big\|_{\mathbf{L}^1(\square_{\iota}, I_M^N)}\leq\frac{L\sqrt{N}}{n}\cdot |DF|(\square_{\iota}).
\end{equation}
On the other hand, since~\eqref{Ftilde-def1} implies
 \[
\big|\widetilde{F}_{\iota}-\overline{F}_{\iota}\big|\leq\frac{\sqrt{N} M}{n},
 \]
it follows that
\begin{equation}
\label{ftilde-est3}
\big\|\widetilde{F}-\overline{F}_{\iota}\big\|_{\mathbf{L}^1(\square_{\iota}, I_M^N)}\leq\frac{\sqrt{N}M\cdot (2L)^N}{n^{N+1}}.
\end{equation}
Combining the estimates~\eqref{Poincare}, \eqref{ftilde-est3}, 
and observing that, by definition~\eqref{monotone-class},
$F\in \mathcal{F}_{[L,M,C]}$ implies $|DF|(I_L^N)\leqslant C$, 
we obtain
\begin{equation}
\label{ftilde-est4}
\begin{aligned}
\big\|\widetilde{F}-F\big\|_{\mathbf{L}^1(I_L^N, I_M^N)}&\leqslant 
\sum_{\iota}\Big(\big\|\widetilde{F}-\overline{F}_{\iota}\big\|_{\mathbf{L}^1(\square_{\iota}, I_M^N)}+
\big\|F-\overline{F}_{\iota}\big\|_{\mathbf{L}^1(\square_{\iota}, I_M^N)}
\Big)
\\
\noalign{\medskip}
&\leqslant \frac{\sqrt{N}M\cdot (2L)^N}{n}+\frac{L\sqrt{N}}{n}\cdot |DF|(I_L^N)
\\
\noalign{\medskip}
&\leqslant \frac{\sqrt{N}}{n}\Big(M(2L)^N+LC\Big),
\end{aligned}
\end{equation}
proving~\eqref{est2}.
\\
\quad\\
{\bf 2.} We introduce now a set of piecewise constant functions sharing the properties~\eqref{ftilde-range}, \eqref{increasing1}
of $\widetilde F$. Namely, letting $J_{M,n}$ be the set in~\eqref{ftilde-range}, we define
\begin{equation}
\label{G-set}
\begin{aligned}
\mathcal{G}_n:=\bigg\{
G\in \mathbf{L}^{\infty}\big(I_L^N, J_{M,n}^N\big)\ \Big|\ \ &G \ \text{is constant on every } \ \square_{\iota}, \
\iota=(\iota_1,...,\iota_N)\in\lbrace{0,...,n-1\rbrace}^{N}
\\
&\text{and} \qquad\quad \iota_i<n-1\quad
\Longrightarrow\quad
G^{i}_{\iota}\geqslant G^{i}_{\iota+e_i} \bigg\},
\end{aligned}
\end{equation}
where $G_{\iota}$ stands for the value of $G$ on the cube $\square_{\iota}$.
Observe that for every $F\in \mathcal{F}_{[L,M,C]}$, letting $\widetilde F$
be the map defined in~\eqref{Ftilde-def2}, one has $\widetilde F\in \mathcal{G}_n$.
Moreover, setting for any given $G\in \mathcal{G}_n$
\begin{equation}
\label{U-cover}
U(G):=\bigg\{F\in \mathbf{L}^1\big(I_L^N, I_M^N\big)\ \Big| \
 \big\|F- G\big\|_{\mathbf{L}^1} \leqslant   \gamma^1_{_{[L,M,C,N]}}\cdot\frac{1}{n}
\bigg\},
\end{equation}
with $\gamma^1_{_{[L,M,C,N]}}$ as in~\eqref{gamma1-def}, because of~\eqref{est2} we have that
$F\in \mathcal{F}_{[L,M,C]}$ implies $F\in U(\widetilde F)$.
Hence, the set
\begin{equation*}
\mathcal{U}:=\Big\{U(G)\ \big|\ G\in \mathcal{G}_n
\Big\}
\end{equation*}
provides an ${\bf L}^1$ covering of $\mathcal{F}_{[L,M,C]}$ with sets of diameter $2 \gamma^1_{_{[L,M,C,N]}}\cdot\frac{1}{n}$.
Thus, 
taking 
\begin{equation}
\label{n-eps-def}
n=\Bigg\lfloor\frac{\gamma^1_{_{[L,M,C,N]}}}{\varepsilon}\Bigg\rfloor+1,
\end{equation}
%
we deduce that 
\begin{equation}
\label{GtildeFbound1}
\mathcal{N}_{\varepsilon}\Big(\mathcal{F}_{[L,M,C]}\ \big|\ \mathbf{L}^1\big(I_L^N, I_M^N\big)\Big)\leqslant 
\text{Card}(\mathcal{G}_n).
\end{equation}
Observe that, given any fixed $\overline i \in\lbrace{1,...,N\rbrace}$,
the set of piecewise constant scalar functions
\begin{equation}
\label{G-set2}
\begin{aligned}
\mathcal{G}_n^{\overline i}:=\bigg\{
g\in \mathbf{L}^{\infty}\big(I_L^N, J_{M,n}\big)\ \Big|\ \ &g \ \text{is constant on every } \ \square_{\iota}, \
\iota=(\iota_1,...,\iota_N)\in\lbrace{0,...,n-1\rbrace}^{N}
\\
&\text{and} \qquad\quad \iota_{\overline i}<n-1\quad
\Longrightarrow\quad
g_{\iota}\geqslant g_{\iota+e_{\overline i}} \bigg\},
\end{aligned}
\end{equation}
defined with the same notations as in~\eqref{G-set}, is independent of the choice of $\overline i$. 
Thus, we deduce from~\eqref{GtildeFbound1} that there holds
\begin{equation}
\label{GtildeFbound2}
\mathcal{N}_{\varepsilon}\Big(\mathcal{F}_{[L,M,C]}\ \big|\ \mathbf{L}^1\big(I_L^N, I_M^N\big)\Big)\leqslant 
\big(\text{Card}(\mathcal{G}_n^{\overline i}\,) \big)^{\!N}.
\end{equation}
Next, we define the set
\[
\mathcal{J} :=\Big\{\iota\in \lbrace{0,...,n-1\rbrace}^N\ \big|\ \iota_{\overline i}=-M\Big\},
\]
which collects all the labels of squares $\Box_{\iota}$ with boundary intersecting the hyperplane $x_{\overline i}=-M$.
Consider the set of decreasing $n$-tuples of elements of the set $J_{M,n}$ in~\eqref{ftilde-range}
\[
\mathcal{K} :=
\bigg\{(a_0, a_1, \dots, a_{n-1})\in (J_{M,n})^n\ \, \Big|\ \,
\!-M\!+\Big(n-{\textstyle\frac{1}{2}}\Big)\cdot{\textstyle\frac{2M}{n}}\geqslant a_0\geqslant\cdots\geqslant a_{n-1}\geqslant 
\!-M\!+{\textstyle\frac{M}{n}}
\bigg\}.
\]
By the definition~\eqref{G-set2} we deduce that
\begin{equation}
\label{cardG-est1}
\text{Card}(\mathcal{G}_n^{\overline i}\,)\leqslant \big(\text{Card}(\mathcal{K})\big)^{\text{Card}(\mathcal{J})}.
\end{equation}
Observe that $\mathcal{K}$ has the same cardinality as the set of decreasing $n$-tuples of nonnegative
integers smaller than $n-1$. By elementary combinatorial arguments it thus follows that, if $n\geq 6$,
one has 
$\text{Card}(\mathcal{K})\leqslant \binom{2n} n\leqslant 2^{2n}$ (e.g.~see~\cite[proof of Lemma 3.1]{DLG}).
Therefore, since $\text{Card}(\mathcal{J})=n^{N-1}$, we derive from~\eqref{GtildeFbound2}, \eqref{cardG-est1}
the upper bound
\begin{equation}
\label{NF-lowbound1}
\mathcal{N}_{\varepsilon}\Big(\mathcal{F}_{[L,M,C]}\ \big|\ \mathbf{L}^1\big(I_L^N, I_M^N\big)\Big)\leqslant 
(2^{2n})^{(N n^{N-1})}=2^{2N n^{N}}.
\end{equation}
Then, for every $0<\varepsilon \leqslant \frac{1}{5}\cdot\gamma^1_{_{[L,M,C,N]}}$, 
with $\gamma^1_{_{[L,M,C,N]}}$ as in \eqref{gamma1-def},
taking $n$ as in~\eqref{n-eps-def} we recover from~\eqref{NF-lowbound1} the estimate
\begin{equation}
\label{NF-lowbound2}
\mathcal{N}_{\varepsilon}\Big(\mathcal{F}_{[L,M,C]}\ \big|\ \mathbf{L}^1\big(I_L^N, I_M^N\big)  \Big)\leqslant 2^{\frac{\gamma_{_{[L,M,C,N]}}}{\varepsilon^{N}}},
\end{equation}
where $\gamma_{_{[L,M,C,N]}}:= 2N\big(2\cdot \gamma^1_{_{[L,M,C,N]}}\big)^{\!N}$. 
In turn, \eqref{NF-lowbound2} yields~\eqref{monotone-upbound}, completing the proof
of the proposition.
\end{proof}
\medskip

Relying on Proposition~\ref{Entropy1} we now establish  an 
upper bound on the $\varepsilon$-entropy in $\mathbf{W}^{1,1}$ for the class of semiconcave functions
introduced in~\eqref{semiconc-def}.
\begin{proposition}\label{estSC}
Given $L,M,K>0$, let $\mathcal{SC}_{[L,M,K]}$
be the set defined in~\eqref{semiconc-def}.
Then, for $\varepsilon>0$ sufficiently small, there holds 
\begin{equation}
\label{upper-entr-sc-1}
 \mathcal{H}_{\varepsilon}\Big(\mathcal{SC}_{[L,M,K]}\ \big|\ \mathbf{W}^{1,1}\big(\mathbb{R}^N\big)\Big)\leqslant 
 \gamma^{\mathcal{SC}}_{_{[L,M,K,N]}}\cdot\frac{1}{\varepsilon^N}
\end{equation}
where
 \begin{equation}\label{Consta1}
 \gamma^{\mathcal{SC}}_{_{[L,M,K,N]}}:=
 \omega_N^N \cdot \Big(4N\cdot\big(1+ M+(K+1)L\big)\Big)^{\!4N^2}\,.
 \end{equation}
 %
\end{proposition}
\begin{proof}
Given $L, K>0$, let us define the map $\mathcal{T}_K:  \mathbf{L}^1(I_L^N) \to  \mathbf{L}^1(\R^N)$,
$I_L:=\,]\!-\!L, L[\,$,
that associates to any $f\in \mathbf{L}^1(I_L^N)$ the function
\begin{equation}
\label{Tk-map}
\mathcal{T}_K f(x) := 
\begin{cases}
f(x)+\frac{K}{2}|x|^2\quad &\text{if}\qquad x\in I_L^N\,,
\\
\noalign{\smallskip}
\  0 &\text{otherwise,}
\end{cases}
\end{equation}
and then consider the class of concave functions 
\begin{equation}
\label{C-klm-class}
\mathcal{C}_{[L,M,K]} := \Big\lbrace{f\in\mathbf{W}^{1,1}(I_L^N) \ \big{|}\ \mathcal{T}_K f\in\mathcal{SC}_{[L,M,K]} \Big\rbrace}.
\end{equation}
The definition~\eqref{C-klm-class} must be understood in the sense that a function $f\in \mathbf{W}^{1,1}(I_L^N)$
is an element of $\mathcal{C}_{[L,M,K]}$ if there exists $\overline f$ almost everywhere equal to $f$
such that $\mathcal{T}_K \overline f\in\mathcal{SC}_{[L,M,K]}$.
Notice that, for any $g\in \mathcal{SC}_{[K,L,M]}$, letting $\widetilde g:=\big(f-\frac{K}{2}|\cdot|^2\,\big)\!\!\!\restriction_{I_L^N}$
denote the restriction of $f-\frac{K}{2}|\cdot|^2\,$ to $I_L^N$, 
recalling definitions~\eqref{Cclass}, \eqref{semiconc-def} one has $\widetilde g\in \mathbf{W}^{1,1}(I_L^N) $ and 
$\mathcal{T}_K \widetilde g = g$. Thus, $\mathcal{T}_K$ is a surjective map from $\mathcal{C}_{[L,M,K]}$
into $\mathcal{SC}_{[K,L,M]}$, and hence
every given  $\varepsilon$-covering $\mathcal{B}=\{B_\alpha\}_\alpha$
of $\mathcal{C}_{[K,L,M]}$
in $\mathbf{W}^{1,1}(I_L^N)$ 
yields an $\varepsilon$-cove\-ring $\{\mathcal{T}_K(B_\alpha)\}_\alpha$
of $\mathcal{SC}_{[K,L,M]}$
in $\mathbf{W}^{1,1}(\R^N)$ 
with the same cardinality. This implies that
\begin{equation}
\label{entr-C-entr-SC}
\mathcal{H}_{\varepsilon}\Big(\mathcal{SC}_{[L,M,K]}\ \big|\ \mathbf{W}^{1,1}\big(\mathbb{R}^N\big)\Big)\leqslant
\mathcal{H}_{\varepsilon}\Big(\mathcal{C}_{[L,M,K]}\ \big|\ \mathbf{W}^{1,1}\big(I_L^N\big)\Big)\,.
\end{equation}
Therefore, in order to establish~\eqref{upper-entr-sc-1}, it will be sufficient to show
\begin{equation}
\label{upper-entr-sc-2}
\mathcal{H}_{\varepsilon}\Big(\mathcal{C}_{[L,M,K]}\ \big|\ \mathbf{W}^{1,1}\big(I_L^N\big)\Big)
\leqslant \gamma^{\mathcal{SC}}_{_{[L,M,K,N]}}\cdot\frac{1}{\varepsilon^N}\,.
\end{equation}
\\
{\bf 1.} 
Towards a proof of~\eqref{upper-entr-sc-2}  observe 
 that, for any given $f\in \mathcal{C}_{[L,M,K]}$, by definitions~\eqref{Cclass}, \eqref{semiconc-def}, \eqref{C-klm-class},
and applying Proposition~\ref{prop2},
there is a representative of $f$, that we still denote $f$, so that
\begin{itemize}
\item[$(i)$] the map $x\mapsto f(x)+\frac{K}{2}|x|^2$ is semiconcave in $I_L^N$ with constant $K$
and has zero trace on~$\partial I_L^N$;
\item[$(ii)$] the superdifferential $D^+ f$ is a 
monotone decreasing multifunction in $I_L^N$;
\item[$(iii)$] there holds 
\begin{equation}
\label{M1-est}
\big\|\nabla f\big\|_{\mathbf{L}^{\infty}(I_L^N)}\leqslant  M_1,
\end{equation}
where 
\begin{equation}
\label{M1}
M_1:= M+K\sqrt{N}L\,. 
\end{equation}
\end{itemize}
By Theorem~\ref{Pro-Semi}-$(iii)$, in turn~\eqref{M1-est} yields
\begin{equation}
\label{M1-est2}
\mathrm{diam}\big(D^+f(I_L^N)\big)\leqslant 2M_1\,.
\end{equation}
Then, relying on $(ii)$ and on~\eqref{M1-est2}, and 
invoking Proposition \ref{BV-bound}, we obtain
\begin{equation}
\label{M1-est3}
\big|D^2 f\big| (I_L^N) \leqslant C_1\,,
\end{equation}
with
\begin{equation}
\label{C1-def}
C_1:= 2^{\frac{3N}{2}}\cdot N^{\big(\frac{N}{2}+2\big)}\cdot\omega_N\cdot (M+(K+1)L)^N\,,
\end{equation}
where $|D^2 f|$ denotes the total variation of the (matrix-valued) distributional derivative $D^2 f$.
Therefore, if we consider the class of 
monotone multifunctions 
\begin{equation}
\label{DC-C1}
\mathcal{DC}_{[L,M,K]}:= \Big\lbrace{D^+ f\ \Big|\ f\in \mathcal{C}_{[L,M,K]}\Big\rbrace},
\end{equation}
recalling definition~\eqref{monotone-class}, by~\eqref{M1-est}, \eqref{M1-est3} we have
\begin{equation}
\mathcal{DC}_{[L,M,K]}\subset \mathcal{F}_{[L, {M_1}, C_1]},
\end{equation}
and hence there holds
\begin{equation}
\label{entr-DC-entr-F}
\mathcal{H}_{\varepsilon}\Big(\mathcal{DC}_{[L,M,K]}\ \big|\  \mathbf{L}^1(I_L^N, I_{M_1}^N)\Big)
\leqslant
\mathcal{H}_{\varepsilon}\Big(\mathcal{F}_{[L, {M_1}, C_1]}\ \big|\  \mathbf{L}^1(I_L^N, I_{M_1}^N)\Big)\,,
\end{equation}
with $I_{M_1}:=\,]\!-\!M_1, M_1[\,$. Thus, relying on Proposition~\ref{Entropy1}, we find
\begin{equation}
\label{Est4}
\mathcal{H}_{\varepsilon}\Big(\mathcal{DC}_{[L,M,K]}\ \big|\  \mathbf{L}^1(I_L^N, I_{M_1}^N)\Big)
\leq\gamma_{_{[L,M_1,C_1,N]}}\cdot\frac{1}{\varepsilon^N} 
\end{equation}
where
$\gamma_{_{[L,M_1,C_1,N]}}$
is a constant defined as in~\eqref{Glmcn-def}
with $M_1, C_1$ given in~\eqref{M1}, \eqref{C1-def}.
\\
\quad\\
{\bf 2.} Relying on~\eqref{Est4} and invoking the Poincar\'e inequality
for trace-zero $\mathbf{W}^{1,1}$ functions
stated in Section~\ref{subsec:semiconc-monotone}, we shall produce now
an $\varepsilon$-covering of $\mathcal{C}_{[L,M,K]}$ in $\mathbf{W}^{1,1}$ with a
cardinality of order~$\leq  \frac{(LN)^N\gamma_{_{[L,M_1,C_1,N]}}}{\varepsilon^N}$. 
In fact, observe that by property $(i)$ above, for every $f_1, f_2 \in \mathcal{C}_{[L,M,K]}$,
one has $f_1-f_2\in \mathbf{W}^{1,1}_0(I_L^N)$. Hence, applying
the Poincar\'e inequality for $\mathbf{W}^{1,1}_0$ functions stated in Theorem~\ref{poincare},
we get
\begin{equation*}
\label{w11-l1-ineq}
\qquad\quad
\big\|f_2-f_1\big\|_{\mathbf{W}^{1,1}(I_L^N)}\leqslant (2L+1)\cdot\big\|\nabla f_2- \nabla f_1\big\|_{\mathbf{L}^1(I_L^N, I_{M_1}^N)}\,
\qquad\forall~ f_1, f_2 \in \mathcal{C}_{[L,M,K]},
\end{equation*}
so that, for any $f_1, f_2\in \mathcal{C}_{[L,M,K]}$, there holds
\begin{equation}
\label{eps-implic}
\big\|\nabla f_2- \nabla f_1\big\|_{\mathbf{L}^1(I_L^N, I_{M_1}^N)}\leq\frac{\varepsilon}{(2L+1)}
\quad \Longrightarrow\quad
\big\|f_2-f_1\big\|_{\mathbf{W}^{1,1}_0(I_L^N)}\leq\varepsilon\,.
\end{equation}
Next, 
by virtue of the estimate~\eqref{Est4} on the $\mathcal{H}_{\varepsilon'}$
entropy of $\mathcal{DC}_{[L,M,K]}$ with $\varepsilon'=\frac{\varepsilon}{2L+1}$, 
there exist $p$ functions $f_1, f_2, \dots f_p\in \mathcal{C}_{[L,M,K]}$, with
\begin{equation}
\label{p-bound}
p\leq \Big\lfloor 2^{\big(\gamma_{_{[L,M_1,C_1,N]}}\cdot (\frac{2L+1}{\varepsilon})^N\big)} \Big\rfloor,
\end{equation}
so that
$$
\mathcal{DC}_{[L,M,K]}
\subset \bigcup_{l=1}^p B\Big(D^+f_l, \frac{\varepsilon}{(2L+1)}\Big)
$$
where  $B(D^+f_l, \varepsilon)$ denotes the $\mathbf{L}^1(I_L^N, I_{M_1}^N)$-ball
centered at $D^+f_l=\nabla f_l$ (regarded as an element of $\mathbf{L}^1(I_L^N, I_{M_1}^N)$).
Therefore, by definition~\eqref{DC-C1}
and because of~\eqref{eps-implic}, we deduce that
\begin{equation}
\label{Cklm-entr-est}
\mathcal{C}_{[L,M,K]}
\subset \bigcup_{l=1}^p B\big(f_l, \varepsilon\big)
\end{equation}
where  $B(f_l, \varepsilon)$ denotes the $\mathbf{W}^{1,1}(I_L^N)$-ball
centered at $f_l$.
Hence, observing that by~\eqref{Glmcn-def}, \eqref{M1}, \eqref{C1-def}, one has
\begin{equation}
\begin{aligned}
(2L+1)^N\cdot\,&\gamma_{_{[L,M_1,C_1,N]}}= (2L+1)^N \cdot 2^{(N^2+N+1)}\,N^{\big(\frac{N}{2}+1\big)}\cdot\big(M_1 L^N+LC_1\big)^N
\\
\noalign{\smallskip}
&\leq (L+1)^N \cdot 2^{(N+1)^2}\cdot N^{N(N+3)}\cdot\omega_N^N\cdot \big((M+KL) L^N+L(M+(K+1)L)^N\big)^{\!N}
\\
\noalign{\smallskip}
&\leq 2^{(N+1)^2}\cdot N^{N(N+3)}\cdot\omega_N^N\cdot \Big(1+M +(K+2)L\Big)^{N(N+2)}
\\
\noalign{\smallskip}
&\leq \omega_N^N \cdot \Big(4N\cdot\big(1+ M+(K+1)L\big)\Big)^{\!4N^2}\,,
\end{aligned}
\end{equation}
it follows from~\eqref{p-bound}, \eqref{Cklm-entr-est} that there holds~\eqref{upper-entr-sc-2}
with $\gamma^{\mathcal{SC}}_{_{[L,M,K,N]}}$ as in~\eqref{Consta1}, thus completing the proof.
%
\end{proof}
\medskip

\subsection{Conclusion of the proof of Theorem~\ref{upper-lower-estimate}-${\bf (i)}$}

Given, $L,M,T>0$, combining Proposition~\ref{Com-entro} and Proposition~\ref{estSC}
we find that, for $\varepsilon$ sufficiently small, there holds
\begin{equation}
\label{UH-S_T}
\mathcal{H}_{\epsilon}\Big(S_T(\mathcal C_{[L,M]})+T\cdot H(0)\ |\ \mathbf{W}^{1,1}(\mathbb{R}^N)\Big)
\leqslant  \gamma^{\mathcal{SC}}_{_{\big[l_{[L,M,T]},M,\frac{1}{\alpha\,T},N\big]}}\cdot\frac{1}{\varepsilon^N}\,,
\end{equation} 
where 
\begin{equation}
\gamma^{\mathcal{SC}}_{_{\big[l_{[L,M,T]},M,\frac{1}{\alpha\,T},N\big]}} = 
 \omega_N^N \cdot \bigg(4N\cdot\Big(1+ M+\big({1}/{(\alpha\,T)}+1\big)\cdot l_{[L,M,T]}\Big)\bigg)^{\!4N^2},
\end{equation}
with $l_{[L,M,T]}$ as in~\eqref{LT}. This establishes the upper bound~\eqref{Upper-est-H}.
\qed
\bigskip

\section{Lower estimates}
\label{sec:low-est}

\subsection{Part 1: Controllability}

Towards a proof of Theorem~\ref{upper-lower-estimate}-$(ii)$, we shall first show that, 
at every given time $T>0$,
one can represent the semiconcave functions of the the set~\eqref{semiconc-def}
as the values at time $T$ of the Hopf-Lax solutions to~\eqref{HJ} with initial data varying in a set of the form~\eqref{Cclass}
translated by $T\cdot H(0)$,
provided that the semiconcavity constant is sufficiently small. 

\begin{proposition}\label{control-part}
Let $H:\R^N\to \R$ be  a function satisfying  the assumptions~{\bf (H1)-(H2)}
and \linebreak  $\{S_t : \Lip (\mathbb{R}^N) \to \Lip (\mathbb{R}^N)\}_{t\geqslant 0}$ be the semigroup of viscosity solutions 
generated by~\eqref{HJ}. 
Then, given any $L,M,T>0$, for every $m, K>0$ such that 
\begin{equation}
\label{Cond4}
m\leq \min \bigg\{m_0,\, M,\,\frac{L}{4\,\big\|D ^2H(0)\big\|\cdot T}\bigg\}\,,\qquad
K\leq \frac{1}{4\,\big\|D ^2H(0)\big\|\cdot T}\,,
\end{equation}
where $m_0$ is the constant in~\eqref{DH0-upbound}, there holds
\begin{equation}
\label{incl-sc}
\mathcal{SC}_{[L/2,\,m,\,K]}\subset S_T(\mathcal C_{[L,M]})+T\cdot H(0)\,,
\end{equation}
where $\mathcal{SC}_{[L/2,\,m,\,K]}$, $\mathcal C_{[L,M]}$ denote sets defined 
as in \eqref{semiconc-def}, \eqref{Cclass}, respectively.
%
%
\end{proposition}
The proof of Proposition~\ref{control-part} is based on the  lemma 
 below, which shows
that a solution of~\eqref{HJ} with a semiconvex initial
condition preserves  the semiconvexity on a given time interval, provided the semiconvexity
constant of the initial data is sufficiently small in absolute value.
\begin{lemma}
\label{le:semiconvexity}
In the same setting of Proposition~\ref{control-part}, given $M, T>0$, let $u_0$ be a
semiconvex function with semiconcavity constant $-K$. Assume that $K>0$ satisfies
%
\begin{equation}
\label{Cond5}
K\leq \frac{1}{2\alpha_M\,T}
\qquad\mbox{where}
\qquad \quad \alpha_M:=\sup_{|p|\leqslant M}\big\|D ^2H(p)\big\|\,,
\end{equation}
and  $\Lip[u_0]\leqslant M$. Then, the following hold true.
\begin{itemize}
\item[$(i)$] $x\mapsto S_t u_0(x)$ is semiconvex for all $t\in [0,T]$.
\item[$(ii)$] $(t,x)\mapsto S_t u_0(x)$ is a $C^1$ classical solution of~\eqref{HJ} on~$]0,T]\times \R^N$.
\end{itemize}
\end{lemma}
\medskip

\noindent
{\it Proof of Proposition~\ref{control-part}.} \
We will show that any element $\psi$ of the set on the left-hand side of~\eqref{incl-sc}
can be obtained as the value at time $T$
of a classical solution to~\eqref{HJ} by reversing the direction of time, and constructing
a backward solution to~\eqref{HJ} that starts at time $T$ from $\psi$. Namely, given
\begin{equation}
\label{psi-in-sc}
\psi\in \mathcal{SC}_{[L/2,\,m,\,K]}\,,
\end{equation}
set
\begin{equation}
\label{in-data-back}
w_0(x):=-\psi(-x)
\qquad\forall x\in\mathbb{R}^N,
\end{equation}
and consider the viscosity solution 
$S_t w_0(x)$  of~\eqref{HJ}.
Because of~\eqref{psi-in-sc}, \eqref{in-data-back}, and
by definitions~\eqref{Cclass}, \eqref{semiconc-def},
we have $w_0\in  \mathcal{C}_{[L/2,\,m]}$.
Moreover,  recalling~\eqref{LT}, \eqref{DH0-upbound}, thanks to~\eqref{Cond4}, \eqref{Cond5}, and to the assumption~{\bf (H2)}, one has 
\begin{equation}
\label{L-est1}
\begin{aligned}
l_{[L/2,m,T]} &\leq L/2 +T\cdot \sup_{|p|\leqslant m_0}\big\|D ^2H(p)\big\| \cdot m
\leq L\,.
\end{aligned}
\end{equation}
Hence, applying Lemma~\ref{ProV}, 
we find
%
\begin{gather}
\label{prop1-STw0}
S_T w_0\in \Lip(\mathbb{R}^N)\,,\qquad
\Lip[ S_T w_0]\leqslant m\,,
\\
\noalign{\smallskip}
\label{prop2-STw0}
S_T w_0(x)=-T\cdot H(0) \qquad\forall~x\in\mathbb{R}^N\setminus [-L,L]^N\,.
\end{gather}
%
On the other hand, notice that by~\eqref{DH0-upbound}, \eqref{semiconc-def}, \eqref{Cond4}, \eqref{psi-in-sc} $\psi$ 
is a semiconcave function with semiconcavity
constant~$K$ 
satisfying~\eqref{Cond5}, with $m$ in place of $M$. Then, it follows from~\eqref{in-data-back} that $w_0$ is semiconvex with semiconvexity constant~$-K$.
Thus, applying Lemma~\ref{le:semiconvexity}, we deduce that 
$S_t w_0(x)$ is a $C^1$ classical solution of~\eqref{HJ} on~$]0,T]\times \R^N$,
continuous on $[0,T]\times \R^N$, and
with initial data $w_0$.
In turn, this implies that the function
\begin{equation}
\label{w-def}
w(t,x):=S_t w_0(x)+T\cdot H(0)
\qquad\quad(t,x)\in[0,T]\times\mathbb{R}^N\,,
\end{equation}
is also a $C^1$ classical solution of~\eqref{HJ} on~$]0,T]\times \R^N$,
continuous on $[0,T]\times \R^N$, and that satisfies
\begin{gather}
\label{prop3-STw0}
w(T,\cdot)\in \Lip(\mathbb{R}^N)\,,\qquad
\Lip[ w(T,\cdot)]\leqslant m\,,
\\
\noalign{\smallskip}
\label{prop4-STw0}
w(T,x)=0 \qquad\forall~x\in\mathbb{R}^N\setminus [-L,L]^N\,.
\end{gather}
Next, notice that, by the above observations, the function 
\begin{equation}\label{eq:u}
u(t,x):=-w(T-t,-x)
\end{equation}
is a $C^1$ classical solution of~\eqref{HJ} on~$]0,T[\,\times \R^N$,
continuous on $[0,T]\times \R^N$. Thus, recalling Remark~\ref{visos-sol-differentiab}, we deduce that $u(t,x)$ is a viscosity
solution of~\eqref{HJ} on $[0,T]\times \R^N$, so that, by the uniqueness property $(i)$ of the semigroup map $S_t$,  one has 
\begin{equation}
\label{u=St}
\qquad
u(t,x)=S_t u_0, \qquad u_0:=u(0,\cdot)
\qquad\quad\forall~(t,x)\in [0,T]\times\mathbb{R}^N.
\end{equation}
Moreover, by virtue of~\eqref{Cond4}, \eqref{prop3-STw0}, \eqref{prop4-STw0},
and by definition~\eqref{Cclass}
it follows that
\begin{equation}
\label{in-data-u}
u_0= -w(T,-\cdot)\in \mathcal{C}_{[L,M]}\,.
\end{equation}
On the other hand, because of~\eqref{in-data-back}, \eqref{w-def}, \eqref{eq:u},
there holds
\begin{equation}
\label{Stu0}
S_T u_0 (x)=
-w_0(-x)-T\cdot H(0)=\psi(x) -T\cdot H(0)\qquad \forall~x\in\mathbb{R}^N\,.
\end{equation}
Hence, \eqref{in-data-u}-\eqref{Stu0} together yield 
\begin{equation}
\label{psi-in-att-set}
\psi\in S_T(\mathcal C_{[L,M]})+T\cdot H(0),
\end{equation}
which completes the proof of the proposition, being $\psi$ an arbitrary element satisfying~\eqref{psi-in-sc}.
\qed
\begin{remark}\rm
The above proof shows that, for any given $\psi\in \mathcal{SC}_{[L/2,\,m,\,K]}$,
one actually finds $u_0\in\mathcal C_{[L,m]}$ which is semiconvex with constant $-\frac{1}{\alpha T}$
so that $\psi = S_T u_0 + T\cdot H(0)$. Indeed, by~\eqref{w-def}, \eqref{eq:u}, \eqref{u=St},
 one has $-u_0=S_T w_0(-\cdot)+T\cdot H(0)$, which  is semiconcave in $\R^N$ with constant~$\frac{1}{\alpha T}$ thanks to Lemma~\ref{ProV}.
\end{remark}

\noindent
{\it Proof of Lemma~\ref{le:semiconvexity}.} \
Observe first that, by Lemma~\ref{ProV}, the map $x\mapsto S_t u_0(x)$ is semiconcave for any fixed $t\in \,]0,T]$.
Therefore,
once we establish the property $(i)$ of Lemma~\ref{le:semiconvexity},
invoking Proposition~\ref{co:smoothness} we immediately deduce that also the property $(ii)$ holds.
On the other hand, by the semiconcavity of $S_t u_0$, we know that $S_t u_0$
is a continuous map.
Hence, in oder to prove the lemma, we only have to show that, for any fixed $t\in \,]0,T]$, 
the map $u(t,x):= S_t u_0(x)$ satisfies the lower bound
\begin{equation}\label{Uni2}
u(t,x+h)+u(t,x-h)-2u(t,x)\geqslant -K_M\cdot |h|^2\qquad\forall x,h\in\mathbb{R}^N\,,
\end{equation}
for some constant $K_M>0$, depending on $K$ and $M$. 
\\
\quad\\
{\bf 1.} 
Towards a proof of~\eqref{Uni2},
fix $x,h\in\mathbb{R}^N$, and let $y_h^{\pm}$ be a minimizer of the function
\begin{equation}
\label{eq-min}
y\mapsto t\cdot H^*\bigg(\frac{x\pm h-y}{t}\bigg)+w_0(y)\qquad(y\in\R^N)\,,
\end{equation}
where $H^*$ denotes the Legendre transform of $H$.
Then, recalling the Hopf-Lax formula~\eqref{Hopf}, one has
\begin{equation}
\label{uxh-eq1}
u(t,x\pm h)=t\cdot H^*\bigg(\frac{x\pm h-y_h^{\pm}}{t}\bigg)+u_0(y_h^{\pm}).
\end{equation}
Moreover, since $y_h^{\pm}$ is a minimizer of~\eqref{eq-min}, by the definition of the subdifferential
in~\eqref{sup-sub-diff}
it follows that there will be some 
\begin{equation}
\label{ph-subdiff}
p_h^{\pm}\in D^-u_0(y_h^{\pm}),
\end{equation}
such that
\begin{equation}\label{PR1}
\nabla H^*\bigg(\frac{x\pm h-y_h^{\pm}}{t}\bigg)=
p_h^{\pm}.
\end{equation}
%
Since $\Lip[u_0]\leqslant M$,
applying Theorem~\ref{Pro-Semi}-$(iii)$ it follows that 
\begin{equation}
\label{ph-bound}
|p_h^{\pm}|\leq M.
\end{equation}
On the other hand, the Hopf-Lax formula implies that
\begin{equation}
\label{ux-eq2}
u(t,x)\leqslant t\cdot H^*\Bigg(\frac{x-\frac{y_h^{+}+y_h^{-}}{2}}{t}\Bigg)+u_0\bigg(\frac{y_h^{+}+y_h^{-}}{2}\bigg).
\end{equation}
Hence, combining~\eqref{uxh-eq1}, \eqref{ux-eq2}, we find
\begin{equation}
\label{uxh-eq3}
\begin{aligned}
u(t,x+h)+u(t,x-h)&-2u(t,x)\geqslant u_0(y_h^+)+u_0(y_h^-)-2u_0\bigg(\frac{y_h^{+}+y_h^{-}}{2}\bigg)+
\\
&+t\cdot\Bigg[H^*\bigg(\frac{x+h-y_h^+}{t}\bigg)+H^*\bigg(\frac{x-h-y_h^-}{t}\bigg)-2\cdot H^*\Bigg(\frac{x-\frac{y_h^{+}+y_h^{-}}{2}}{t}\Bigg)\Bigg].
\end{aligned}
\end{equation}
Since $H^*$ is convex and $u_0$ is semiconvex with  constant $-K$, we obtain from~\eqref{uxh-eq3}
the inequality
\begin{equation}\label{Conv1}
u(t,x+h)+u(t,x-h)-2u(t,x)\geqslant -\frac{K}{4}\cdot \big|y_h^+-y_h^-\big|^2.
\end{equation}
\\
{\bf 2.} 
In order to recover the estimate~\eqref{Uni2} from \eqref{Conv1} we need to provide an upper bound \linebreak on~$|y_h^+-y_h^-|^2$.
To this end, observe first that, in view of (\ref{PR1}), one has
\begin{equation}\label{Dist1}
\Bigg\langle \nabla H^*\bigg(\frac{x+h-y_h^+}{t}\bigg)-\nabla H^*\bigg(\frac{x-h-y_h^-}{t}\bigg),-\frac{y_h^+-y_h^-}{t}\Bigg\rangle=
\Bigg\langle p_h^+-p_h^-,\,-\frac{y_h^+-y_h^-}{t}\Bigg\rangle\,.
\end{equation}
On the other hand, owing to the semiconvexity of $u_0$ and by virtue of~\eqref{sup-sub-diff-equal}, \eqref{ph-subdiff},
we can apply 
Proposition~\ref{prop2}-$(i)$  to get
\begin{equation*}\label{ESt4}
\bigg\langle \!\!-p_h^++p_h^-,\,\frac{y_h^+-y_h^-}{t}\bigg\rangle\leqslant \frac{K}{t}\cdot \big|y_h^+-y_h^-\big|^2,
\end{equation*} 
which, together with~\eqref{Dist1}, yields
\begin{equation}
\label{ESt4}
\Bigg\langle \nabla H^*\bigg(\frac{x+h-y_h^+}{t}\bigg)-\nabla H^*\bigg(\frac{x-h-y_h^-}{t}\bigg),-\frac{y_h^+-y_h^-}{t}\Bigg\rangle
\leqslant \frac{K}{t}\cdot \big|y_h^+-y_h^-\big|^2\,.
\end{equation}
Next, observe that  there holds
\begin{equation}
\begin{aligned}
\label{eq:convexity}
\Bigg\langle \nabla H^*\bigg(\frac{x+h-y_h^+}{t}\bigg)-\nabla H^*&\bigg(\frac{x-h-y_h^-}{t}\bigg),-\frac{y_h^+-y_h^-}{t}\Bigg\rangle=
\\
&=\int_{0}^{1}\bigg\langle D^2H^*(z_\tau)\bigg(\frac{2h-(y_h^+-y_h^-)}{t}\bigg),-\frac{y_h^+-y_h^-}{t}\bigg\rangle d\tau\,,
\end{aligned}
\end{equation}
where 
\begin{equation}
\label{ztau-def}
z_\tau=\tau\cdot\frac{x+h-y_h^+}{t}+(1-\tau)\cdot\frac{x-h-y_h^-}{t}.
\end{equation}
Now, 
relying on~\eqref{HH*-grad}, 
(\ref{PR1}), \eqref{ph-bound}, and   assumption {\bf (H2)},
we get
\begin{equation*}
|z_\tau|\leqslant \sup_{|p|\leqslant M}\big|(\nabla H^*)^{-1}(p)\big|
\leqslant \sup_{|p|\leqslant M}\big\|D ^2H(p)\big\|\cdot M =\alpha_M\cdot M\qquad\forall~\tau\in [0,1]\,,
\end{equation*}
where $\alpha_M$ denotes a constant defined as in~\eqref{Cond5}.
Hence, one has
%
\begin{equation}
\label{low-est1}
 \int_{0}^{1}\bigg\langle D^2H^*(z_\tau)\Big(\frac{2h}{t}\Big),-\frac{y_h^+-y_h^-}{t}\bigg\rangle d\tau\geqslant -2\beta_M\cdot\frac{|h|\cdot|y_h^+-y_h^-|}{t^2},
 \end{equation}
where 
\begin{equation}
\label{KM-def}
\beta_M:=\sup_{|p|\leqslant \alpha_M\cdot M}\big\|D^2H^*(p)\big\|.
\end{equation}
On the other hand, notice that
 the definition of $\alpha_{M}$ in \eqref{Cond5} implies $DH^2(p)\leqslant\alpha_{M}\cdot \mathbb{I}_{N}$ for  $|p|\leqslant M$. 
 Thus, recalling~\eqref{HH*-grad},
 we deduce that, for every $q=(\nabla H^*)^{-1}(p)$, with $|p|\leqslant M$, one has 
 \[
 D^2H^*(q)=\big(D^2H(p)\big)^{-1}\geqslant\frac{1}{\alpha_M}\cdot  \mathbb{I}_{N},
 \]
 which, by~(\ref{PR1}), \eqref{ph-bound}, and because of the definition~\eqref{ztau-def} of $z_\tau$, implies
 \[
 D^2H^*(z_\tau)\geqslant\frac{1}{\alpha_M}\cdot  \mathbb{I}_{N}
 \qquad\quad\forall~\tau\in [0,1]\,.
 \]
 Therefore, we find
 \begin{equation}
\label{low-est2}
 \int_{0}^{1}\bigg\langle D^2H^*(z_\tau)\Big(-\frac{y_h^+-y_h^-}{t}\Big),-\frac{y_h^+-y_h^-}{t}\bigg\rangle d\tau
 \geqslant\frac{1}{\alpha_M}\cdot\frac{|y_h^+-y_h^-|^2}{t^2}.
 \end{equation}
Combining \eqref{eq:convexity} with
the lower bounds~\eqref{low-est1}, \eqref{low-est2}, one obtains 
 \begin{equation}
 \label{ESt5}
 \begin{aligned}
 \Bigg\langle \nabla H^*\bigg(\frac{x+h-y_h^+}{t}\bigg)-\nabla H^*&\bigg(\frac{x-h-y_h^-}{t}\bigg),-\frac{y_h^+-y_h^-}{t}\Bigg\rangle
 \\
 &\geqslant\frac{1}{\alpha_M}\cdot\frac{|y_h^+-y_h^-|^2}{t^2}-2\beta_M\cdot\frac{|h|\cdot|y_h^+-y_h^-|}{t^2}.
 \end{aligned}
 \end{equation}
 \\
{\bf 3.} 
The upper bound (\ref{ESt4}) together with the lower bound (\ref{ESt5}) yields
 \[2\beta_M\cdot|h|\cdot|y_h^+-y_h^-|\geq\bigg(\frac{1}{\alpha_M}-Kt\bigg)|y_h^+-y_h^-|^2.
 \]
 In turn, recalling (\ref{Cond4}), from the above inequality it follows that
 \[
 |y_h^+-y_h^-|\leqslant 4\beta_M\alpha_M\cdot |h|\,.
 \]
Finally, using this last estimate, it is immediate to deduce \eqref{Uni2}  from (\ref{Conv1}),
with $K_M=4K\beta^2_M\alpha^2_M$, where $\alpha_M, \beta_M$ are defined in~\eqref{Cond5} and \eqref{KM-def},
respectively. 
This completes the proof of the lemma.
\qed
\bigskip

\subsection{Part 2: Lower compactness estimates on a class of bump functions}
We provide here a lower bound on the $\varepsilon$-entropy for the class of semiconcave functions
$\mathcal{SC}_{[L,M,K]}$ introduced in~\eqref{semiconc-def}.
\begin{proposition}\label{Lower estimate}
Given any $L,M,K>0$,  
for every
\begin{equation}
\label{eps1-bound}
0<\varepsilon\leq \min\{K,\,M\}\cdot \frac{\omega_N\, L^{N}}{(N+1)\,2^{N+8}}\,,
\end{equation}
there holds
\begin{equation}
\label{low-entr-sc}
\mathcal{H}_{\varepsilon}\Big(\mathcal{SC}_{[L,M,K]}\ \big|\ \mathbf{W}^{1,1}(\mathbb{R}^N)\Big)\geqslant 
\beta^{\mathcal{SC}}_{_{[L,K,N]}}\cdot\frac{1}{\varepsilon^N},
\end{equation}
where 
\begin{equation}\label{eq:Gamma}
\beta^{\mathcal{SC}}_{_{[L,K,N]}}:= \frac{1}{8\cdot\ln 2}\cdot \bigg(\frac{K\, \omega_N\, L^{N+1}}{48(N+1)\,2^{N+1}}\bigg)^{\!\!N}.
\end{equation}
\end{proposition}
\begin{proof} 
The proof is given in three steps.
We shall first define a prototype $C^1$ bump function with Lipschitz continuous gradient. Next, 
we shall consider a class of semiconcave functions $\mathcal{U}_n$ defined as  superpositions of such a bump function, localized  
on the $N$-dimensional cubes of a partition of the domain $[-L, L]^N$. Finally, we shall derive
an optimal
lower bound on the covering number $\mathcal{N}_\varepsilon\big(\mathcal{U}_n\, | \, \mathbf{W}^{1,1}(\mathbb{R}^N)\big)$
for a suitable choice of $n$, which then yields~\eqref{low-entr-sc}.

\medskip

\noindent
\textit{Step 1: construction of a bump function.} \\
Consider the continuously differentiable function $c:[0,1]\rightarrow\mathbb{R}$ defined by 
\begin{equation}\label{beta}
c(r)=
\begin{cases}
\hspace{.cm}
\big(\frac{1}{4}\big)^2-\int_{0}^{r}\big(\frac{1}{4}-|\frac{1}{4}-s|\big)ds
\quad &\text{if}\qquad r\in\big[0,\frac{1}{2}\big],
\vspace{.3cm}
\\
\qquad 0\quad &\text{if}\qquad r\in \big[\frac{1}{2},1\big]\,.
\end{cases}
\end{equation}
Then, we compute
\begin{equation*}
c'(r)=
\begin{cases}
\hspace{.cm}
\big|\frac{1}{4}-r\big|-\frac{1}{4}
\quad &\text{if}\qquad r\in\big[0,\frac{1}{2}\big],
\vspace{.3cm}
\\
\qquad 0\quad &\text{if}\qquad r\in \big[\frac{1}{2},1\big]\,.
\end{cases}
\end{equation*}
Thus, $c'$ is Lipschitz continuous with Lipschitz constant $1$ and there holds
\begin{equation}
\label{cprime-lipsch}
\big|c'(r)\big|\leq r\qquad\ \forall~r\in[0,1]\,.
\end{equation}
Moeover, one has
\begin{equation}
\label{c-cprime-bounds}
\|c\|_{\mathbf{L}^{\infty}([0,1])}\leq\frac{1}{16},\qquad\qquad
\|c'\|_{\mathbf{L}^{\infty}([0,1])}\leq\frac{1}{4}\,.
\end{equation}
%
%
We now proceed to  construct  our bump function $b:[-L,L]^N\rightarrow\mathbb{R}$ as  follows:
\begin{equation}
\label{eq:bump}
b(x)=
\begin{cases}
\hspace{.cm}
 \frac{KL^2}{6}\, {c}\Big(\frac{|x|}{L}\Big)
\quad &\text{if}\qquad x\in B\big(0,\frac{L}{2}\big)
\vspace{.3cm}
\\
\hspace{.cm}
\qquad 0\quad &\text{if}\qquad x\in [-L,L]^N\setminus B\big(0,\frac{L}{2}\big)\,.
\end{cases}
\end{equation} 
One can check that 
\begin{equation}
\label{nabla-b1}
\nabla b(x)=\frac{KL}{6}\, c'\Big(\frac{|x|}{L}\Big)\,\frac{x}{|x|}\qquad \forall~x\in[-L,L]^N\setminus 0\qquad \mathrm{and}\qquad \nabla b(0)=0\,.
\end{equation}
Thus, because of~\eqref{c-cprime-bounds}, there holds
\begin{equation}
\label{b-nablab-ellein}
\|b\|_{\mathbf{L}^{\infty}([-L,L]^N)}\leq\frac{KL^2}{96}\,\qquad\qquad
\|\nabla b\|_{\mathbf{L}^{\infty}([-L,L]^N)}\leq\frac{KL}{24}\,.
\end{equation}
Furthermore, since $c'$ is $1$-Lipschitz, observing that
\begin{equation*}
\bigg|\frac{y}{|y|}-\frac{x}{|x|}\bigg|\leq \frac{2\big|y-x\big|}{|x|}
\qquad\forall~x,y\neq 0\,,
\end{equation*}
and relying on~\eqref{cprime-lipsch}, \eqref{nabla-b1}, it follows that $\nabla b$ is Lipschitz continuous with  constant $K/2$ in $[-L,L]^N$.
On the other hand, observing that 
\begin{equation}
\label{nabla-b-0}
\nabla b(x) =0\qquad\qquad\forall~x\in [-L,L]^N\setminus B(0,\tfrac{L}{2})\,,
\end{equation}
a
straightforward computation shows that
\begin{eqnarray*}
\big\|\nabla b\big\|_{\mathbf{L}^{1}([-L,L]^N)}&=&
\frac{KL}{6}\,\int_{B(0,L/2)}\Big|c'\Big(\frac{|x|}{L}\Big)\Big|dx\\
&=&\frac{KL^{N+1}}{6}\,\int_{B(0,1/2)}\big|c'(|x|)\big|dx=\frac{N KL^{N+1}\omega_{N}}{6}\int_{0}^{1}|c'(r)|r^{N-1}dr\\
&=&\frac{KL^{N+1}\omega_{N}}{6}\cdot\frac{2^N-1}{2(N+1)4^N}\,.
 \end{eqnarray*} 
 Thus, setting
 \begin{equation}\label{Norm-bump}
\beta_{_{[L,K,N]}}:=
\frac{KL^{N+1}\omega_{N}}{12}\cdot\frac{2^N-1}{(N+1)4^N},
 \end{equation}
 we have 
 \begin{equation}
 \label{nablab-elle1}
\big\|\nabla b\big\|_{\mathbf{L}^{1}([-L,L]^N)}= \beta_{_{[L,K,N]}}.
 \end{equation}
%
%
%
%
%

Now, given any positive integer $n\in\mathbb{N}$, let us consider the continuously differentiable 
function $b_n:\big[-\!\frac{L}{n},\frac{L}{n}\big]^N\rightarrow\mathbb{R}$ defined as 
\begin{equation}
\label{eq:bump-n}
b_n(x)=\frac{ b(nx)}{n^2}\,,\quad\forall x\in \Big[\!-\!\tfrac{L}{n},\tfrac{L}{n}\Big]^N.
\end{equation}
Thus, by~\eqref{eq:bump} one has 
\begin{equation}
\label{bn-supp}
b_n(x)=0\qquad\ \text{if}\qquad x\in \big[-\!\tfrac{L}{n},\tfrac{L}{n}\big]^N\setminus B\big(0,\tfrac{L}{2n}\big).
\end{equation}
Noting that 
$\nabla b_n(x)=\frac{1}{n}\cdot \nabla b(nx)$ for  $x\in [-\frac{L}{n},\frac{L}{n}]^N$, 
and relying on~\eqref{b-nablab-ellein}, \eqref{nablab-elle1}, one can easily check that
\begin{equation}
\label{bn-prop}
\big\|\nabla b_n\big\|_{\mathbf{L}^{\infty}( [-\frac{L}{n},\frac{L}{n}]^N)}\leq\frac{KL}{24n},\qquad\quad
\big\|\nabla b_n\big\|_{\mathbf{L}^{1}( [-\frac{L}{n},\frac{L}{n}]^N)}=\frac{1}{n^{N+1}}\,\big\|\nabla b\big\|_{\mathbf{L}^{1}([-L,L]^N)}=\frac{\beta_{_{[L,K,N]}}}{n^{N+1}}.
\end{equation}
Moreover, 
since $\nabla b$ is Lipschitz continuous with  constant $K/2$, we have that $\nabla b_n$ is also  Lipschitz 
continuous with  constant $K/2$. 
By Remark~\ref{semiconc-prop}
this implies that $b_{n}$ and $-b_{n}$  are semiconcave functions with  constant $K$.
\bigskip

\noindent
\textit{Step 2: a class of semiconcave functions defined as  superpositions of bump functions.} \\
For any integer $n\geqslant 1$ let us divide $[-L,L]^N$ into $n^{N}$ cubes of side $\frac{2L}{n}$
as in the proof of Proposition~\ref{Entropy1}. More precisely, we shall use the notation
\begin{equation}
[-L,L]^N= \bigcup_{\iota\in\lbrace{1,...,n\rbrace}^N} \square_{\iota}\,,
\end{equation} 
where $\iota=(\iota_1,...,\iota_N)\in\lbrace{1,...,n\rbrace}^N$ is a multiindex and
\begin{equation*}
\square_{\iota}:=(-L,...,-L)+\tfrac{L}{n}\,\iota+\Big[{\textstyle -\frac{L}{n},\frac{L}{n}}\Big]^N
\end{equation*}
is an $N$-dimensional cube centered at
$
x_{\iota}:=(-L,...,-L)+\frac{L}{n}\,\iota\,.
$
 Let us now adapt our bump function $b_n$ in~\eqref{eq:bump-n} to the cube $\square_{\iota}$ defining
\begin{equation*}
b^{\iota}_n(x)=
\begin{cases}
\hspace{.2cm}b_n(x-x_{\iota})\quad&\text{if} \qquad x\in\square_{\iota},
\vspace{.2cm}
\\
\hspace{.2cm}
\quad 0
\quad&\text{if} \qquad x\in\mathbb{R}^{N}\setminus\square_{\iota}\,.
\end{cases}
\end{equation*}
One can easily verify that the continuously differentiable function $b^\iota_n : \R^n\to \R$ shares the same properties of $b_n$.
In particular, by~\eqref{bn-supp}, \eqref{bn-prop}, there holds:
\begin{enumerate}
\item[(i)] $b^{\iota}_n(x)=0$ for all $x\in\mathbb{R}^{N}\backslash B\big(x_{\iota},\frac{L}{2n}\big)$,
\item[(ii)] $\|\nabla b^{\iota}_n\|_{\mathbb{L^{\infty}}(\mathbb{R}^{N})}\leq\frac{KL}{24n}$ \ and \  $\|\nabla b^{\iota}_n\|_{\mathbf{L}^1(\mathbb{R}^{N})}= \frac{\beta_{_{[L,K,N]}}}{n^{N+1}}$,
\item[(iii)] $b^{\iota}_n$ and $-b^{\iota}_n$ are semiconcave with  constant $K$. 
\end{enumerate}
Next, we proceed to construct a class of semiconcave functions in the set $\mathcal{SC}_{[L,M,K]}$, defined as combinations of the
bump functions $b^{\iota}_n$.
Namely, consider the set of $n^N$-tuples
\begin{equation*}
\Delta_n=\Big\lbrace{\delta=(\delta_{\iota})_{\iota\in\lbrace{1,...,n\rbrace}^N}\ |\ \delta_{\iota}\in\lbrace{-1,1\rbrace}\Big\rbrace},
\end{equation*}
and, for every $\delta=(\delta_{\iota})_{\iota\in\lbrace{1,...,n\rbrace}^N}\in \Delta_n$, define the function
$u_{\delta}:\R^N\to\R$ by setting
\begin{equation}
\label{uiota-def}
u_{\delta}:=\sum_{\iota\in\lbrace{1,...,n\rbrace}^N}\delta_{\iota}\, b^{\iota}_n\,.
\end{equation}
Observe that, by properties $(i)$-$(iii)$ above, every  function
$u_{\delta}$
has support contained in  $[-L,L]^N$, is semiconcave with semiconcavity constant $K$, and satisfies 
$
\|\nabla u\|_{\mathbf{L}^{\infty}(\mathbb{R}^{N})}\leqslant M$ provided that  
\begin{equation}
\label{K-ineq1}
n\geq\frac{KL}{24M}\,.
\end{equation}
Therefore, recalling definition~\eqref{semiconc-def}, one has
\begin{equation}
\label{Un-SC}
\mathcal{U}_n:=\Big\lbrace{u_{\delta}\ \big|\ \delta\in\Delta_n\Big\rbrace}
\subset \mathcal{SC}_{[K,L,M]}\,,
\end{equation}
for all such $n$.
Hence, in order to establish~\eqref{low-entr-sc}, it will be sufficient to show that there holds
\begin{equation}
\label{low-entr-sc-2}
\mathcal{H}_{\varepsilon}\Big(\mathcal{U}_n\ \big|\ \mathbf{W}^{1,1}(\mathbb{R}^N)\Big)\geqslant 
\beta^{\mathcal{SC}}_{_{[L,K,N]}}\cdot\frac{1}{\varepsilon^N}
\end{equation}
for every $\varepsilon$ sufficiently small and for a suitable choice of $n$  satisfying~\eqref{K-ineq1}.
\bigskip

\noindent
\textit{Step 3: estimate of the $\varepsilon$-entropy for superpositions of bump functions
by a combinatorial argument.} \\
Towards an estimate of the covering number $\mathcal{N}_{\varepsilon}\big(\mathcal{U}_n\ \big|\ \mathbf{W}^{1,1}(\mathbb{R}^N)\big)$,
fix $\bar{\delta}\in\Delta_n$, and let us define the set of $n^N$-tuples 
$$
\mathcal{I}_{\bar{\delta},n}(\varepsilon)=\Big\lbrace{\delta\in\Delta_n\ |\ \|\nabla u_{\bar{\delta}}-\nabla u_{\delta}\|_{\mathbf{L}^{1}(\mathbb{R}^{N})}\leq\varepsilon\Big\rbrace}\,.
$$
Notice that, by construction, 
the cardinality of the set $\mathcal{I}_{\bar{\delta},n}(\varepsilon)$ is independent of the choice of~$\bar{\delta}\in\Delta_n$.
Let us denote it by
$$
C_n(\varepsilon):=\card\big(\mathcal{I}_{\bar{\delta},n}(\varepsilon)\big).
$$
Moreover, any element of an $\varepsilon$-cover in $\mathbf{W}^{1,1}$ of $\mathcal{U}_n$ contains at most
$C_n(2\varepsilon)$ functions of $\mathcal{U}_n$.
Hence, since the cardinality of $\mathcal{U}_n$ is the same as the cardinality of $\Delta_n$, which is 
$\card(\Delta_n)=2^{n^N}$, it follows that the number of sets in an $\varepsilon$-cover in $\mathbf{W}^{1,1}$ of $\mathcal{U}_n$
is at least
\begin{equation}
\label{low-entr-sc-2}
\mathcal{N}_{\varepsilon}\Big(\mathcal{U}_n\ \big|\ \mathbf{W}^{1,1}(\mathbb{R}^N)\Big)\geqslant 
\frac{2^{n^N}}{C_n(2\varepsilon)}\,.
\end{equation}
Aiming at an upper bound on $C_n(2\varepsilon)$, observe that for any given pair $\delta, \overline\delta\in \Delta_n$,
one has
\begin{equation}\label{W-11}
\big\|\nabla u_{\bar{\delta}}-\nabla u_{\delta}\big\|_{\mathbf{L}^{1}(\mathbb{R}^{N})}= d\big(\bar{\delta},\delta\big)\cdot 2\big\|\nabla b_n\big\|_{\mathbf{L}^{1}( [-\frac{L}{n},\frac{L}{n}]^N)}\,,
\end{equation}
where 
\begin{equation}
\label{d-def}
d\big(\bar{\delta},\delta\big):=\card\Big(\big\lbrace{\iota\in\big\lbrace{1,...,n\rbrace}^N\ |\ \bar{\delta}_{\iota}\neq\delta_{\iota}\big\rbrace}\Big)\,.
\end{equation}
Thus, relying on~\eqref{bn-prop}, \eqref{W-11}, we deduce that
\begin{equation}\label{W-12}
d(\bar{\delta},\delta)\leq\frac{n^{N+1}}{\beta_{_{[L,K,N]}}}\cdot \varepsilon
\qquad\Longleftrightarrow\qquad
\big\|\nabla u_{\bar{\delta}}-\nabla u_{\delta}\big\|_{\mathbf{L}^{1}(\mathbb{R}^{N})}\leq 2\varepsilon\,. 
\end{equation}
Hence, performing a standard combinatorial computation of the number of $n^N$-tuples that differ for a given number of entries,
we find
\begin{equation}
\label{Cn}
C_n(2\varepsilon)\leq\sum_{l=0}^{\Big\lfloor \frac{n^{N+1}}{\beta_{_{[L,K,N]}}}\cdot\varepsilon \Big\rfloor} \binom{\:\;n^N}{l}.
\end{equation}
Next, observe that if $X_1,...,X_{n^N}$ are independent random variables with uniform Bernoulli distribution
$\mathbb{P}(X_i=1)=\mathbb{P}(X_i=0)=\frac{1}{2}$, then, for any $k\leqslant n^N$, one has 
\begin{equation}\label{cb}
\sum_{l=0}^k \binom{\;\;n^N}{l}=2^{n^{N}}\cdot\mathbb{P}\Big(X_1+...+X_{n^N}\leqslant k\Big)\,.
\end{equation}
Now, set $S_{n^N}=X_1+...+X_{n^N}$, and  recall Hoeffding's inequality (\cite[Theorem 2]{Hoeffding}) which guarantees that, for any  $\mu>0$, 
\begin{equation}
\label{H}
{\mathbb P} \big(S_{n^N} - {\mathbb{E}}[S_{n^N}] \leqslant -\mu\big)  \leqslant \exp \left( - \frac{2 \mu^{2}}{n^N}\right),
\end{equation}
where ${\mathbb{E}}[S_{n^N}]$ denotes the expectation of $S_{n^N}$. 
Since ${\mathbb{E}}[S_{n^N}]=\frac{n^N}{2}$, taking 
$\mu=\frac{n^N}{2}-\Big\lfloor \frac{n^{N+1} }{\beta_{_{[L,K,N]}}}\cdot\varepsilon \Big\rfloor$ and assuming 
\begin{equation}
\label{eps-bound1}
n\leqslant\frac{\beta_{_{[L,K,N]}}}{2\,\varepsilon}\,,
\end{equation}
from (\ref{Cn}), (\ref{cb}) and (\ref{H}) it follows that
\begin{equation}
\label{low-entr-sc-3}
\begin{aligned}
C_n(2\varepsilon)&\leqslant
2^{n^N}\,\exp\left(-\frac{2\Big(\frac{n^N}{2}-\Big\lfloor \frac{n^{N+1} }{\beta_{_{[L,K,N]}}} \cdot\varepsilon\Big\rfloor\Big)^2}{n^N}\right)\\
&\leqslant 2^{n^N}\,\exp\Bigg(-\frac{n^N}{2}\, \bigg(1-\frac{2n\,\varepsilon}{\beta_{_{[L,K,N]}}}\bigg)^2\Bigg).
\end{aligned}
\end{equation}
In turn, \eqref{low-entr-sc-3} together with~\eqref{low-entr-sc-2}, yields
\begin{equation}
\label{low-entr-sc-4}
\mathcal{N}_{\varepsilon}\Big(\mathcal{U}_n\ \big|\ \mathbf{W}^{1,1}(\mathbb{R}^N)\Big)\geqslant 
\exp\Bigg(\frac{n^N}{2}\, \bigg(1-\frac{2n\,\varepsilon}{\beta_{_{[L,K,N]}}}\bigg)^2\Bigg)
\end{equation}
for all $n$ satisfying~\eqref{eps-bound1}.
Now, if we take
\begin{equation}
\label{eps0-bound}
0<\varepsilon\leq 
\min\bigg\{\tfrac{\beta_{_{[L,K,N]}}}{8}, \tfrac{6M\, \beta_{_{[L,K,N]}}}{KL}\bigg\},
\end{equation}
choosing
\begin{equation}\label{choice of n}
n_{\varepsilon}:=\bigg\lfloor \frac{\beta_{_{[L,K,N]}}}{4\,\varepsilon} \bigg\rfloor +1,
\end{equation}
one easily check that $n_{\varepsilon}$ satisfies both bounds~\eqref{K-ineq1}, \eqref{eps-bound1}.
Hence, relying on~\eqref{Un-SC}, \eqref{low-entr-sc-4}, we find the lower bound
\begin{equation}
\label{low-entr-sc-5}
\begin{aligned}
\mathcal{N}_{\varepsilon}\Big(\mathcal{SC}_{[K,L,M]}\ \big|\ \mathbf{W}^{1,1}(\mathbb{R}^N)\Big)
&\geqslant\mathcal{N}_{\varepsilon}\Big(\mathcal{U}_{n_\varepsilon}\ \big|\ \mathbf{W}^{1,1}(\mathbb{R}^N)\Big)\\
\noalign{\smallskip}
&\geqslant 
\exp\bigg(\frac{n_{\varepsilon}^N}{8}\bigg)\geqslant\exp\Bigg(\frac{\beta^N_{_{[L,K,N]}}}{8\,(4\varepsilon)^N}\Bigg),
\end{aligned}
\end{equation}
for all $\varepsilon$ satisfying~\eqref{eps0-bound}.
In turn, this estimate yields~\eqref{low-entr-sc}  for all $\varepsilon$ satisfying~\eqref{eps1-bound},
taking $\log_2$ of both sides of~\eqref{low-entr-sc-5}
and observing that, by~\eqref{eq:Gamma}, \eqref{Norm-bump}, one has
\begin{equation*}
\beta_{_{[L,K,N]}}\geqslant 
\frac{K\, \omega_N\, L^{N+1}}{24(N+1)\,2^{N}},
\qquad\quad
\frac{1}{8\, \ln 2}\cdot \Bigg(\frac{\beta_{_{[L,K,N]}}}{4}\Bigg)^{\!\!N}\!\!\geqslant
\beta^{\mathcal{SC}}_{_{[L,K,N]}}\,.
\end{equation*}
\end{proof}
\subsection{Conclusion of the proof of Theorem~\ref{upper-lower-estimate}-${\bf (ii)}$}

Given, $L,M,T>0$, combining Proposition~\ref{control-part} and Proposition~\ref{Lower estimate}
we find that, for every
\begin{equation}
\label{eps2-bound}
0<\varepsilon\leq \min\bigg\{M,\,\frac{1}{4\,\big\|D ^2H(0)\big\|\cdot T},\,\frac{L}{4\,\big\|D ^2H(0)\big\|\cdot T}\bigg\}\cdot \frac{\omega_N\, L^{N}}{(N+1)\,2^{N+8}}\,,
\end{equation}
there holds
\begin{equation}
\label{Lower-est-H2}
\mathcal{H}_{\epsilon}\Big(S_T(\mathcal C_{[L,M]})+T\cdot H(0)\ \big|\ \mathbf{W}^{1,1}(\mathbb{R}^N)\Big)\geqslant 
\beta^{\mathcal{SC}}_{_{\big[\frac{L}{2},\,\frac{1}{4\,\|D ^2H(0)\|\cdot T},\,N\big]}}\cdot\frac{1}{\varepsilon^N},
\end{equation} 
where
\begin{equation}
\beta^{\mathcal{SC}}_{_{\big[\frac{L}{2},\,\frac{1}{4\,\|D ^2H(0)\|\cdot T},\,N\big]}}
=
\frac{1}{8\cdot\ln 2}\cdot \bigg(\frac{\omega_N\, L^{N+1}}{3(N+1)\,2^{(2N+8)}\cdot\|D ^2H(0)\|\cdot T}\bigg)^{\!\!N}.
\end{equation}
This establishes the lower bound~\eqref{Lower-est-H}.
\qed
\bigskip
\section*{Acknowledgements}
This work was partially supported by the National Group for Mathematical Analysis and Probability (GNAMPA) of the Istituto Nazionale di Alta Matematica ``Francesco Severi'' (INdAM),  the  CNRS-INdAM European Research Group (GDRE) on Control of Partial Differential Equations (CONEDP), and the European Union
$7th$ Framework Programme [FP7-PEOPLE-2010-ITN] under grant agreement n.264735-SADCO.
Fabio Ancona was partially supported by the Miur-Prin 2012 Project "Nonlinear Hyperbolic Partial Differential Equations, 
Dispersive and Transport Equations: theoretical and applicative aspects" 
and by 
Fondazione CaRiPaRo Project "Nonlinear Partial Differential Equations: models, analysis, and control-theoretic problems".
\bigskip

\end{document}